\documentclass[final,times,authoryear,1p]{elsarticle}

\usepackage{amsmath,amsfonts,amssymb}
\usepackage{amsthm}
\usepackage{graphicx}
\usepackage{multirow}
\usepackage{color}
\usepackage{framed}

\usepackage{fancyhdr}
\usepackage{mathtools}

\newtheorem{theorem}{Theorem}[section]

\newtheorem{corollary}[theorem]{Corollary}

\newtheorem{remark}[theorem]{Remark}

\usepackage{amssymb}

\journal{Transportation Research Part C}

\begin{document}

\begin{frontmatter}



 \begin{center}
\textcolor{blue}{ARTICLE LINK:  http://www.sciencedirect.com/science/article/pii/S0968090X15001345
\\  PLEASE CITE THIS ARTICLE AS\\ 
Han, K., Liu, H., Gayah, V., Friesz, T.L., Yao, T., 2015. A robust optimization approach for dynamic traffic signal control with emission considerations. Transportation Research Part C, DOI: 10.1016/j.trc.2015.04.001.}
 \line(1,0){469}
 \end{center}

\title{A robust optimization approach for dynamic traffic signal control with emission considerations}


\author[ic]{Ke Han\corref{cor}}
\ead{k.han@imperial.ac.uk}

\author[ie]{Hongcheng Liu}
\ead{hql5143@psu.edu}

\author[cee]{Vikash V. Gayah}
\ead{gayah@engr.psu.edu}

\author[ie]{Terry L. Friesz} 
\ead{tfriesz@psu.edu}

\author[ie]{Tao Yao}
\ead{tyy1@engr.psu.edu}

\cortext[cor]{Corresponding author}

\address[ic]{Department of Civil and Environmental Engineering, Imperial College London, United Kingdom.}

\address[ie]{Department of Industrial and Manufacturing Engineering, Pennsylvania State University, USA.}

\address[cee]{Department of Civil and Environmental Engineering, Pennsylvania State University, USA.}

\begin{abstract}
We consider an analytical signal control problem on a signalized network whose traffic flow dynamic is described by the Lighthill-Whitham-Richards (LWR) model \citep{Lighthill and Whitham, Richards}. This problem explicitly addresses traffic-derived emissions as side constraints. We seek to tackle this problem using a  mixed integer mathematical programming approach. Such a class of problems, which we call LWR-Emission (LWR-E), has been analyzed before to certain extent. Since mixed integer programs are practically efficient to solve in many cases \citep{BCR}, the mere fact of having integer variables is not the most significant challenge to solving LWR-E problems;  rather, it is the presence of the potentially nonlinear and nonconvex emission-related constraints/objectives that render the program computationally expensive.

To address this computational challenge, we proposed a novel reformulation of the LWR-E problem as a {\it mixed integer linear program} (MILP). This approach relies on the existence of a statistically valid macroscopic relationship between the aggregate emission rate and the vehicle occupancy of the same link. This relationship is approximated with certain functional forms and the associated uncertainties are handled explicitly using robust optimization (RO) techniques. The RO allows emissions-related constraints and/or objectives to be reformulated as linear forms under mild conditions. To further reduce the computational cost, we employ the {\it link transmission model} to describe traffic dynamics with the benefit of fewer (integer) variables and less potential traffic holding. The proposed MILP explicitly captures vehicle spillback,  avoids traffic holding, and simultaneously minimizes travel delay and addresses emission-related concerns. 
\end{abstract}


\end{frontmatter}

\section{\label{Intro}Introduction}

Traffic signals tend to be the primary focus of urban traffic control and management strategies since they generally serve as the most frequent and restrictive bottlenecks on urban streets. Over time, the implementation of traffic signal control has evolved greatly: from simple fixed-time plans based on historical data and updated infrequently throughout the day to adaptive control systems that update continuously in response to real-time traffic information. The performance of a particular strategy depends on several factors: the optimization procedure employed to select signal timings, underlying model used to predict the evolution of traffic dynamics and the objective function considered in the optimization procedure.

Here, we distinguish between two types of optimization procedures: (1) heuristic approaches, such as those developed with feedback control, genetic algorithms and fuzzy logic \citep{KZMT, Zhang}; and (2) exact approaches, such as those arising from mathematical control theory and  mathematical programming. Although useful for very large and complex optimization problems, heuristic approaches suffer from a failure to provide optimal solutions. Instead, these are more appropriate when exact approaches are computationally intractable. However, {\it mixed integer programs} (MIPs) have been used extensively in the signal control literature and are of particular interest due to their tractable for smaller networks. For example, \cite{Improta and Cantarella} formulated and solved the traffic signal control problem for a single road junction as a mixed binary integer program. \cite{Lo 1999a} and \cite{Lo 1999b} formulated the network-level signal control problem as a mixed integer linear program using the {\it cell transmission model} (CTM) \citep{CTM1, CTM2}. In these papers, time-varying traffic demand patterns were incorporated by adopting dynamic signal timing plans.  Such methods were later extended in \cite{Lin and Wang} to capture more realistic features of signalized junctions such as the total number of vehicle stops and signal preemption in the presence of emergency vehicles. Building upon this solid foundation that exists in the literature, this MIP approach will be adopted here.

In most works, the model of traffic dynamics is taken as fixed, which is reasonable for models that accurately describe the critical phenomena observed. In this paper, we consider the Lighthill-Whitham-Richards model \citep{Lighthill and Whitham, Richards}, also known as kinematic wave theory, to describe traffic dynamics on individual links and through signalized junctions. This well-known model is employed as it is one of the most used and trusted traffic flow models currently being used in network optimization procedures today \citep{Aziz, CP, HGPFY, Lin and Wang, LHGFY, Lo 1999, Lo 1999b,  Zhang}. In particular, we employ a {\it link-based kinematic wave model} (LKWM) proposed in  \cite{LKWM}  to capture queue dynamics, shock waves and vehicle spillback, while integrating it with signalized junction models.  In contrast to the cell-based math programming approaches reviewed above, the link-based approach requires fewer spatial variables and eliminates the problem of {\it traffic holding} that arises between two adjacent cells \citep{Z} without using additional binary variables \citep{Lo 1999a}. In this way, the resulting mathematical program is more computationally efficient than the more traditional cell-based approach.

As for objective functions, the majority of adaptive traffic signal control schemes update signal timings to minimize total vehicular delays. Representatives of such signal-control systems are OPAC \citep{Gartner}, RHODES \citep{MH}, SCAT \citep{SD} and SCOOT \citep{HRBR}. Other control strategies seek to minimize delays to a subset of vehicles; e.g., the goal of transit signal priority strategies is to reduce delays for transit vehicles, often to the detriment of those remaining, see \cite{Skabardonis}. More recently, a transit signal priority strategy was proposed to minimize total person delay, which essentially considers a weighted average of vehicular delay using the passenger occupancies of each vehicle as the weights \citep{Christofa}. 

Relatively less attention has been given to vehicular emissions in the optimization of traffic signal timings. The earliest study that includes emissions in signal timing optimization appears to be \cite{Robertson}, but this work relies on macroscopic simulations that do not accurately account for vehicle dynamics at intersections. The efforts that followed either relied on combining detailed emissions models with outputs from microscopic simulations or models \citep{Stevanovic, Li, Lin, Lv} or macroscopic emissions models estimated from data \citep{Aziz, Zhang}. The former approach is more accurate, but relies on computationally intensive simulation-based optimization methods. The latter is useful but as pointed out by a survey paper \citep{Szeto} and the literature therein, the environmental considerations typically result in highly nonlinear and nonconvex constraints and objective functions in the mathematical programming formulation, which also imposes tremendous computational burdens. As a result, heuristic methods, such as one found in \cite{Ferrari, Zhang}, have been used to solve these types of problems. Classical methods such as the inner penalty technique \citep{Yang and Bell} and augmented Lagrangian multiplier technique \citep{YXHM} have also been used, but well-defined exact approaches that account for these non-linear and (potentially) stochastic relationships currently do not exist.

This paper presents a novel approach to circumvent the aforementioned computational challenges by combining traditional objective functions (i.e., minimizing vehicular delays or maximizing vehicle throughput) with emissions considerations. The latter are incorporated using constraints in the optimization procedure (e.g., maximizing throughput subject to some emissions standard that must be met) and these are reformulated as linear functions through the use of  numerical experimentation and robust optimization. This method is made possible by leveraging observed relationships between aggregated emissions rates and vehicles occupancies on a link that arise when certain macroscopic or mesoscopic emission models are employed (e.g., see \citet{Shabihkhani}). Such empirical observations are supported by extensive numerical simulations, as we shall demonstrate below. Detailed description of the simulation and synthetic data is presented in Section \ref{secSim}.

Unfortunately, despite the strong correlation between the aggregated emission rate and certain macroscopic traffic quantities (e.g. link occupancy), there are non-negligible errors associated with such approximation.  Of course, errors and perturbations to a deterministic model can render an optimal solution in the ideal case suboptimal in implementation. A natural approach  to capture uncertainty is by assuming that unknown parameters follow certain probability distributions and by employing the notions and methodologies in stochastic programming. However, such an approach has two main limitations: 1) exact knowledge of error distributions is often difficult to acquire, and 2) stochastic programming is recognized as highly intractable to solve even with linear objective function and linear constraint functions. In view of these challenges, we propose to handle uncertainty in the perspective of robust optimization.

A robust optimization is a distribution-free uncertainty set approach that seeks to minimize the worst-case cost and/or to remain feasible in the worst scenario. Compared to stochastic programming, robust optimization makes no assumption on the underlying distribution of uncertain parameters. Moreover, it has been shown to work as a powerful approximation to stochastic programming and even probabilistic models with significantly reduced computational cost \citep{bn1998, bn1999, bn2000, Bertsimas etal2011a, Bertsimas etal2011b, Bandi, Rikun}. Although solutions to robust optimization problems can be relatively conservative, the conservatism is adjustable with the flexibility of choosing uncertainty sets \citep{bersim2004}. A comprehensive review of robust optimization is provided by \cite{Bertsimas etal2011a}.

Due to the nonlinear and nonconvex nature of emission-related constraints and/or objectives, signal optimization problems with emission considerations are very difficult to solve when formulated as mathematical programs. This paper proposes a practical and effective way to reformulate this problem by invoking a robust optimization approach based on macroscopic emission models. Through numerical simulations we uncover well-defined macroscopic relationships between the link aggregate emission rate and the link occupancy, which is then utilized to re-formulate emission-rated constraints/objectives into mathematically tractable forms. We show that a fair general class of emission-related constraints and objectives can be re-formulated as linear constraints with the utilization of dual variables. Effectively, the signal optimization problem with emission considerations are formulated as mixed integer linear programs (MILPs). These MILPs not only capture realistic traffic dynamics that exist on signalized networks such as shock waves and car spillback, but also address nonlinear and nonconvex emissions constraints/objectives in a mathematically tractable way. Moreover, they can be solved by commercial solvers fairly efficiently in a nearly mathematically tractable way \citep{BCR}. The proposed solution method is tested using a synthetic experiment to demonstrate its performance.

The methodological framework proposed in this paper is potentially transferrable to traffic control problems formulated as mathematical programs, in which vehicle emission, fuel consumption, or safety are within the purview of the traffic operator. For example, vehicle fuel consumption can be modeled in a similar way as emission based on various operational modes of a moving car such as cruise, acceleration, deceleration, and idle \citep{Barth}. On the other hand, traffic risk index has been derived based on statistical analysis of historical information on accidents and macroscopic traffic data including traffic flow, occupancy and speed  \citep{HRL}. These problems may be similarly formulated as MILPs if the underlying macroscopic relationship can be approximated by piecewise affine functions.

The rest of this paper is organized as follows.  Section \ref{secLWRMILP} presents a mixed integer linear programming formulation of network signal optimization problems, without any emission considerations. In Section \ref{secSim}, we present details of the numerical simulations that uncover the macroscopic relationships between the aggregated emission rate and the link occupancy.  In  Section \ref{secRO} we utilize the findings made in Section \ref{secSim} to systematically and explicitly derive the robust counterpart of the LWR-E problem based on fairly general assumptions made on the macroscopic relationship. Section \ref{secgeneralization} discusses two generalizations of the LWR-E problem.   Section \ref{secNE} presents a numerical study of the proposed formulation. Finally, Section \ref{secConclusion} provides some concluding remarks.

\section{MILP approach for signal optimization}\label{secLWRMILP}
This section presents the mixed integer linear program (MILP) formulation of the dynamic traffic signal control problem without any consideration for emission, while Sections \ref{secRO} and \ref{secgeneralization} will address emission-related constraints/objectives in detail. As we mentioned in the introduction, a link-based LWR model will be considered, the derivation of which employs the variational method \citep{VT1}. For the conciseness of our presentation, we only recap the key results below and refer the reader to \cite{LKWM} for a detailed derivation and analysis.

The LWR model is based on the following kinematic wave equation:
\begin{equation}\label{LWRPDE}
\partial_t\rho(t,\,x)+\partial_x f\big(\rho(t,\,x))~=~0
\end{equation}
\noindent where $\rho(t,\,x)$ denotes the vehicle density in a spatial-temporal domain; $f(\cdot)$ is the fundamental diagram that describes the macroscopic relationship between vehicle density and flow on the link. Throughout this paper, $f(\cdot)$  is assumed to be triangular of the following form \footnote{The triangular FD is the key to a much simplified variational representation of the solution and to the link-based LWR model considered in this paper.}: 
\begin{equation}
f(\rho)~=~
\begin{cases}
v\rho \qquad & \rho\in[0,\,\rho^c]
\\
-w(\rho-\rho^{jam})\qquad & \rho\in(\rho^c,\,\rho^{jam}]
\end{cases}
\end{equation}
\noindent where $\rho$ denotes vehicle density; $v$ and $w$ are respectively the speeds of the forward- and backward-propagating kinematic waves; $\rho^c$ is the critical density at which the flow is maximized; and $\rho^{jam}$ denotes the jam density. Moreover, we let $C$ be the flow capacity and $L$ be the length of the link. We allow all these quantities and variables to depend on a specific link $I_i$, and will always use subscript `$i$' to indicate such a dependence.

We consider an link expressed as a spatial interval $[a,\,b]$, and ignore for now the subscript `$i$' for notation convenience. We define a binary variable $\bar r(t)$, which indicates the traffic state at the entrance of the link $x=a$: $\bar r(t)=0$ if traffic is free flow and $\bar r(t)=1$ if traffic is congested. A similar notation $\hat r(t)$ is used for the exit of the link. We also define the link inflow $\bar q(t)$ and the link exit flow $\hat q(t)$. The key results of the link-based kinematic wave model, derived from the variational theory, are as follows. 
\begin{align}
\label{latent3}
\bar r(t)&=\begin{cases}
 1,\quad &\hbox{if}\quad \int_0^t \bar q(\tau)\,d\tau ~=~ \int_0^{t-{L\over w}}\hat q(\tau)\,d\tau +\rho^{jam}L\\
0, \quad  &\hbox{if}  \quad  \int_0^t \bar q(\tau)\,d\tau ~<~ \int_0^{t-{L\over w}}\hat q(\tau)\,d\tau +\rho^{jam}L
\end{cases}
\\
\label{latent4}
\hat r(t)&=\begin{cases}
 0,\quad &\hbox{if}\quad \int_0^{t-{L\over v}}\bar q(\tau)\,d\tau ~=~ \int_0^t\hat q(\tau)\,d\tau\\
1, \quad  &\hbox{if}\quad \int_0^{t-{L\over v}}\bar q(\tau)\,d\tau ~>~ \int_0^t\hat q(\tau)\,d\tau
\end{cases}
\end{align}
\noindent In addition, we introduce the notions of {\it demand} and {\it supply} of a link  \citep{LK1999} which, under the assumption of a triangular fundamental diagram, reduce to the following:
\begin{align}
\label{demanddef}
D(t)=\begin{cases}
C ~~  & \hbox{if}~~\hat r(t)= 1
\\
\bar q\left(t- {L\over v}\right)~~ & \hbox{if}~~\hat r(t)=0
\end{cases}
\qquad\qquad 
S(t)=\begin{cases}
C ~~  & \hbox{if}~~\bar r(t)=0
\\
\hat q\left(t-{L\over w} \right)~~ & \hbox{if}~~\bar r(t)=1
\end{cases}
\end{align}

\subsection{Discrete-time formulation of the link dynamic}\label{seclinkdt}
Let us introduce some key discrete-time notations employed in this paper, where the subscript `$i$' indicates association with link $I_i$, and the superscript `$k$' indicates the k-th time interval.

\begin{framed}
\vspace{-0.15 in}
\begin{itemize}

    \item[]{\makebox[1.5cm]{$\bar q_i^k$\hfill}  the flow at which vehicles enter link $I_i$;} 
\vspace{-0.06 in}
    \item[]{\makebox[1.5cm]{$\hat q_i^k$\hfill}  the flow at which vehicles exit link $I_i$;}
    \vspace{-0.06 in}
       \item[]{\makebox[1.5cm]{$\bar r_{i}^k $\hfill}  the binary variable indicating the traffic state at the entrance of $I_i$;} 
\vspace{-0.06 in}
    \item[]{\makebox[1.5cm]{$\hat r_{i}^k$\hfill}  the binary variable indicating the traffic state at the exit of $I_i$;}
    \vspace{-0.06 in}
       \item[]{\makebox[1.5cm]{$S_i^k$\hfill}  the supply of link $I_i$;} 
\vspace{-0.06 in}
    \item[]{\makebox[1.5cm]{$D_i^k$\hfill}  the demand of link $I_i$;}
    \vspace{-0.06 in}
       \item[]{\makebox[1.5cm]{$u_i^k$\hfill} the binary signal control variable for link $I_i$;} 
    
\end{itemize}
\vspace{-0.15 in}
\end{framed}

\noindent For a fixed time step size $\delta t$, we define $\Delta^f_i\doteq \left[{L_i\over v_i\delta t}\right]$, $\Delta^b_i\doteq \left[{L_i\over w_i\delta t}\right]$, where $[x]$ rounds the real number $x$ to the nearest integer \footnote{Rounding these quantities to integers lead to certain numerical errors due to time discretization. The impact of such an approximation (rounding effect) on the solution quality has been investigated in \cite{CSC}.}. We are now ready to state the discrete versions of \eqref{latent3}-\eqref{latent4} as follows.

\begin{align}
\label{panding1}
&\begin{cases}
\displaystyle \delta t\sum_{k=1}^{l-\Delta^b_i} \hat q_i^k -\delta t\sum_{k=1}^{l}\bar q_i^k+\rho^{jam}_iL_i~\leq~\mathcal{M}\,(1-\bar r_i^l)+\varepsilon
\\
\displaystyle \delta t \sum_{k=1}^{l-\Delta^b_i} \hat q_i^k - \delta t \sum_{k=1}^{l}\bar q_i^k+\rho^{jam}_iL_i~>~-\mathcal{M}\,\bar r_i^l +\varepsilon
\end{cases}\qquad\forall i,~~\forall l
\\
\label{panding2}
&\begin{cases}
\displaystyle \delta t \sum_{k=1}^{l-\Delta^f_i} \bar q_i^k -\delta t\sum_{k=1}^{l}\hat q_i^k~\leq~\mathcal{M}\,\hat r_i^l+\varepsilon
\\
\displaystyle \delta t \sum_{k=1}^{l-\Delta^f_i} \bar q_i^k -\delta t \sum_{k=1}^{l}\hat q_i^k~>~\mathcal{M}\,(\hat r_i^l-1)+\varepsilon
\end{cases}\qquad \forall i,~~\forall l
\end{align}
where $\rho^{jam}_i$ and $L_i$ denote respectively the jam density and length of link $I_i$. $\mathcal{M}>0$ is a large constant, and $\varepsilon>0$ is a small constant serving as a cut-off threshold.  Moreover, the demand $D_i^k$ and the supply $S_i^k$, whose continuous-time expressions are given by \eqref{demanddef}, are determined via the following inequalities, where $C_i$ denotes the flow capacity of link $I_i$:
\begin{equation}\label{inqdemand}
\begin{cases}
C_i+\mathcal{M}(\hat r_i^k-1)~\leq~D_i^k~\leq~C_i
\\
\bar q_i^{k-\Delta^f_i}-\mathcal{M}\hat r_i^k~\leq~D_i^k~\leq~\bar q_i^{k-\Delta_i^f}+\mathcal{M}\hat r_i^k 
\end{cases} \qquad\forall i,~~\forall k
\end{equation}
\begin{equation}\label{inqsupply}
\begin{cases}
C_i-\mathcal{M}\bar r_i^k~\leq~S_i^k~\leq~C_i
\\
\hat q_i^{k-\Delta^b_i} + \mathcal{M}(\bar r_i^k-1)~\leq~S_i^k~\leq~\hat q_i^{k-\Delta_i^b}- \mathcal{M}(\bar r_i^k-1) 
\end{cases} \qquad\forall i,~~\forall k
\end{equation}

\subsection{Dynamics at signalized junctions}\label{secAppsj}
In general, signalized junction models vary according to detailed intersection geometry and phasing schemes. For simplicity yet without loss of generality, we consider in this paper the simple junction shown in Figure \ref{figintersection} while referring the reader to \cite{HG} for an elaborated treatment of detailed junction layout, vehicle movements and multiple signal phases. 
\begin{figure}[h!]
\centering
\includegraphics[width=.4\textwidth]{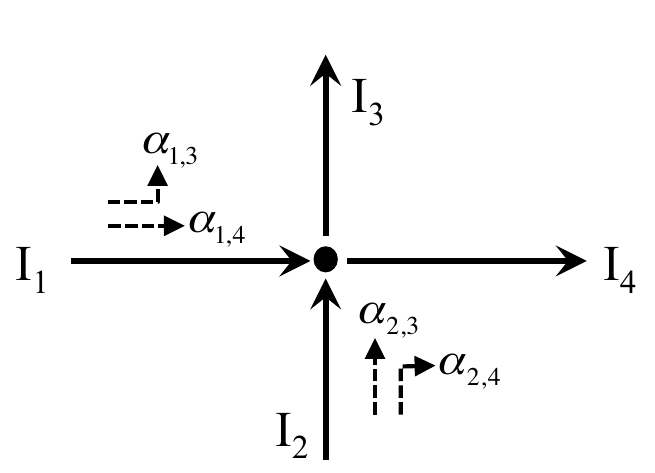}
\caption{Two signalized junctions}
\label{figintersection}
\end{figure}

This junction has two incoming links ($I_1,\,I_2$) and two outgoing links ($I_3,\,I_4$). The vehicle turning percentages, $\alpha_{1,3},\,\alpha_{1,4},\,\alpha_{2,3},\,\alpha_{2,4}$ as shown in Figure \ref{figintersection}, are assumed to be known a priori, and can be usually estimated through historical turn-by-turn vehicle counts. The discrete-time junction dynamic may be written as 
$$
\hat q_i^k~=~\min\left\{  D_i^k~,~ u_i^k\cdot \min\Big\{{S_3^k\over \alpha_{i,3}}~,~{S_4^k\over \alpha_{i,4}}\Big\} \right\}\qquad i=1,\,2,\quad\forall k
$$
$$
\bar q_j^k~=~\alpha_{1,j}\hat q_1^k + \alpha_{2,j}\hat q_2^k \qquad j=3,\,4,\quad\forall k
$$
\begin{remark}
In the above expression, if certain vehicle turning percentage, say $\alpha_{i,3}$, is zero, then the term ${S_3^k\over \alpha_{i,3}}$ is infinity and will be dropped from the ``min" operator. This corresponds to the situation where vehicles discharged from $I_i$ do not enter the downstream link $I_3$, and thus $I_3$ is effectively removed from the junction as far as vehicles from $I_i$ are concerned. Thus, for what follows we always assume positive vehicle turning percentages. Moreover, although we assume here that these percentages are constants, it is a trivial extension to allow them to depend on time (`$k$'). 
\end{remark}
\noindent For $i=1,\,2$, let us define $\zeta_i^k=\min\left\{D_i^k~,~{S_3^k\over \alpha_{i,3}}~,~{S_4^k\over \alpha_{i,4}}\right\}$; then we have $\hat q_i^k~=~u_i^k\cdot \zeta_i^k$, which may be expressed as linear constraints: 
\begin{equation}\label{dtcpeqn5}
\begin{cases}
0~\leq~\hat q_i^k~\leq~\mathcal{M}\,u_i^k\\
\zeta_i^k+\mathcal{M}\,(u_i^k-1)~\leq~\hat q_i^k~\leq~\zeta_i^k
\end{cases}\qquad i=1,\,2,\quad\forall k
\end{equation}
Moreover, $\zeta_i^k$, which is the minimum of three endogenous variables, may be expressed as linear constraints using two additional binary variables, $\xi_i^k$ and $\eta_i^k\in\{0,\,1\}$:
\begin{equation}
\begin{cases}
D_i^k - \mathcal{M}\eta_i^k~\leq~\zeta_i^k~\leq~D_i^k
\\
{S_3^k\over \alpha_{i,3}}-\mathcal{M}\xi_i^k-\mathcal{M}(1-\eta_i^k)~\leq~\zeta_i^k~\leq~{S_3^k\over \alpha_{i,3}}
\\
{S_4^k\over \alpha_{i,4}}-\mathcal{M}(1-\xi_i^k)-\mathcal{M}(1-\eta_i^k)~\leq~\zeta_i^k~\leq~{S_4^k\over \alpha_{i,4}}
\end{cases}\qquad i=1,\,2\qquad \forall k
\end{equation}
It is easy to check that $\zeta_i^k=D_i^k$ when $\eta_i^k=0$; $\zeta_i^k={S_3^k\over\alpha_{i,3}}$ when $\eta_i^k=1$ and $\xi_i^k=0$; $\zeta_i^k={S_4^k\over\alpha_{i,4}}$ when $\eta_i^k=1$ and $\xi_i^k=1$.
Finally, to prevent conflicting traffic streams to be discharged at the same time, we stipulate that 
\begin{equation}\label{uconst}
u_1^k+u_2^k~=~1\qquad \forall k
\end{equation}
Regarding the objective function one has a lot of flexibility in selecting its form as long as linearity is maintained. The following linear form is selected in this paper. 
\begin{equation}\label{mipobj}
\hbox{max}~~\sum_{k=1}^M {1\over k}\sum_{I_i\in \mathcal{I}}\hat q_i^k 
\end{equation}
where $M$ is the total number of time intervals, and $\mathcal{I}$ is a prescribed set of links. For example, $\mathcal{I}$ may be selected to be the set of outgoing links of the network of interest; or $\mathcal{I}$ may be specified as the set of links that one wishes to prioritize in a congested network. Choosing such an objective function ensures that the throughputs on these links are maximized at any instance of time. Again, we emphasize that any type of linear objective function can be selected, depending on the specific application, without affect the MILP formulation.  In summary, the proposed MILP consists of the objective function \eqref{mipobj} and constraints \eqref{panding1} through \eqref{uconst}.

\begin{remark}
No explicit constraints besides \eqref{uconst} are made on the signal control binary variables, which allows the signals to operate using splits and cycle lengths that change dynamically during the course of the control period. In this way, the signal splits and cycle lengths will likely vary with time as traffic flow patterns change. Note that this is the most flexible strategy, but might not be the most realistic in some networks. The methodology can be easily extended to more restrictive but realistic scenarios (e.g., fixed-cycle-and-split, fixed-cycle-dynamic-split, dynamic-cycle-fixed-split, etc.) through the introduction of additional constraints on signal timing parameters. This will not affect the main formulation for minimizing delays subject to emissions constraints. 
\end{remark}

In the MILP formulation, the constant $\mathcal{M}$ can be chosen to be  
$$
\mathcal{M}~\geq~\max\left\{\max\{1,\,T\}\cdot \bar C~,~ {\bar C\over \bar \alpha}\right\}
$$
where $T$ is the length of the time horizon, $\bar C$ is the maximum link flow capacity in the network, $\bar\alpha$ is the smallest vehicle turning percentage in the network (notice that all turning percentages are assumed positive). The cut-off threshold $\varepsilon$ in \eqref{panding1}-\eqref{panding2} should be no greater than $\delta t\cdot \min_i C_i$.

\section{Macroscopic relationship between the emission rate and traffic quantities}\label{secSim}

 As mentioned earlier in the introduction, the key ingredient of the proposed MILP formulation for the LWR-E problem is the existence of a macroscopic relationship between the link's aggregate emission rate (AER) and its vehicle occupancy. In this section we will investigate such a relationship through analytical computations. This requires the modeling of vehicle movements within a link subject to signal controls, and a vehicle emission model that calculates the AER in a way consistent with the vehicle flow dynamics. More specifically, we employ the LWR model with a triangular fundamental diagram to predict and describe the evolution of vehicle density (see Section \ref{secLWRMILP}); on the emission side a modal emission model is considered and detailed below in Section \ref{secmodal}.

\subsection{Simulation setup}\label{subsecsimsetup}

The hypothesized macroscopic relationship is investigated through a battery of simulations that employ the aforementioned LWR model and emission model. The simulations are conducted as follows. We consider an arbitrary link with given length and triangular fundamental diagram. In order to account for various levels of congestion, we randomly generate the demand and supply profiles at the upstream and downstream ends of the link, respectively. The values of the demand and supply are uniformly distributed between zero and the link flow capacity. Moreover, in order to incorporate the effect of the signal timing, we randomly generate sequences of green and red phases to control the link's discharge flow.  Such variability in the signal controls is crucial in the simulation since the resulting macroscopic relationship (if any) will be used across all possible scenarios involving different signal control parameters. With given signal control and upstream/downstream demand profiles, we solve the LWR PDE \eqref{LWRPDE} by applying the variational method \citep{LKWM}. Alternatively, one may solve the same PDE  using the Godunov scheme \citep{Godunov} or the cell transmission model \citep{CTM1}. Once vehicle densities are available in discrete space and time, we then apply a specific emission model to compute the total emission rate and the link occupancy $N(t)$, which is expressed as
\begin{equation}\label{occupancy}
N(t)~=~\int_a^b \rho(t,\,x)dx \qquad \hbox{or}\qquad N(t)~=~\int_0^t \bar q(\tau)d\tau-\int_0^t \hat q(\tau)d\tau
\end{equation}
where the link of interest is expressed as the spatial interval $[a,\,b]$; $\rho(t,\,x)$ is the solution of the LWR PDE \eqref{LWRPDE}, and $\bar q(\cdot)$ and $\hat q(\cdot)$ are the link inflow and exit flow, respectively.

In the next two subsections, we will present in detail the two emission models and the resulting macroscopic relationship obtained from the simulation. The following link parameters are employed in our simulation. 
$$
v=40/3 ~\hbox{meter/s},\quad w=40/9~\hbox{meter/s},\quad \rho^{jam}=0.4~\hbox{vehicle/meter},\quad C=4/3~\hbox{vehicle/second}
$$
where $v$ and $w$ are the forward and backward wave speeds respectively, $\rho^{jam}$ denotes the jam density, and $C$ is the flow capacity. The link length is set to be $L=400$ meters. Note should be taken on the following fact: the macroscopic relationship, as well as the uncertainty set calibrated for the robust optimization presented later, depend on  link-specific parameters such as the fundamental diagram and link length. However, the most relevant fact that we rely upon here is that the relationship or the uncertainty region holds for a range of endogenous factors; this includes: vehicle arrival and discharge rate for the link (both of which are accounted for within our simulation) and most importantly the signal timing at the intersection. The latter fact is especially useful as it allows us to use these functions to derive optimal signal timing plans that both minimize delay and account for constraints in emissions within the network or on individual links.

It should be noted that although this paper considers one specific emission model concerning hydrocarbon due to space limitation, the proposed analytical framework can potentially accommodate a wider range of emission models that are similarly based on vehicle speed, acceleration and deceleration.

\subsection{The modal emission model}\label{secmodal}

We consider a modal emission model, which is based on the operational modes of a vehicle, including idle, steady-state cruise, acceleration and deceleration. This model relies on vehicle trajectories estimated from the LWR model. More specifically, let $\rho(t,\,x)$ be the  solution of the LWR equation \eqref{LWRPDE} on the link of interest; the vehicle speed $v(t,\,x)$  is computed as 
\begin{equation}\label{vcalc}
v(t,\,x)~=~f(\rho(t,\,x))/\rho(t,\,x)
\end{equation} 
The acceleration/deceleration, $a(t,\,x)$, viewed as the derivative of the speed along the trajectories of moving vehicle, is computed as the material derivative in the Eulerian coordinates:
$$
a(t,\,x)~=~{D\over Dt}v(t,\,x)~=~\partial_t v(t,\,x)+v(t,\,x)\cdot\partial_x v(t,\,x)
$$
Since the quantities $\rho(t,\,x)$ and $v(t,\,x)$ are in general non-differentiable, they are approximated in discrete time using finite differences. Let $\{t_i\}$ and $\{x_j\}$ be discrete temporal and spatial grid points, then we have 
\begin{equation}\label{modala}
a(t_i,\,x_j)~=~{v(t_{i+1},\,x_{j})-v(t_{i-1},\, x_j)\over 2 \delta t} + v(t_i,\,x_j)\cdot {v(t_i,\,x_{j+1}) -v(t_i,\,x_{j-1})\over 2\delta x}
\end{equation}
where $\delta t$ and $\delta x$ are the time and spatial steps, respectively.

Following the power-demand emission model proposed by \cite{Post}, the overall instantaneous total power demand $Z$ (in kilowatt) for a vehicle with mass $m$ in (kilogram) is given by 
\begin{equation}\label{ztot}
Z~=~(0.04\,v+0.5\times 10^{-3}v^2+10.8\times 10^{-6}v^3)+{m\over 1000}{v\over 3.6}\left({a\over 3.6}+9.81\sin\theta\right)
\end{equation}
where the above quantity is in kilowatts and $\theta$ denotes the road grade. The reader is also referred to \cite{Barth} for an alternative description of the power demand function based on velocity and acceleration. \cite{Post} also propose the following model of hydrocarbon emissions rate for vehicles based on field experiments:
\begin{equation}\label{modalr}
r(t)~=~\begin{cases}
52.8+ 4.2Z \qquad & Z~>~0
\\
52.8 \qquad & Z~\leq~0
\end{cases}
\end{equation}
where the emission rate $r(t)$ is in grams/hour, and $Z$ is in kilowatt. Following the previous discussion, we now compute the aggregate emission rate (AER) in discrete time as
$$
AER(t_i)~=~\delta x\sum_{j}\rho(t_i,\,x_j) r(t_i,\,x_j)
$$
where $r(t_i,\,x_j)$ is the emission rate of a vehicle corresponding to the $j$-th cell at time $t_i$, calculated using \eqref{vcalc}-\eqref{modalr}.  The resulting scatter plot of the link occupancy (LO) vs. the aggregate emission rate (AER) is shown in Figure \ref{figModal_emission}, which is based on 42,000 simulation runs. From this figure we  can observe a well-defined macroscopic relationship between these two quantities, although the points are relatively scattered, indicating potential errors associated with such an approximation. Overall, Figure \ref{figModal_emission} shows a positive correlation between the AER and the LO. This can be intuitively explained using the following two observations. (1) The total emission rate on the link level is in general expected to show a growing trend with an increased number of vehicles on this link. Although vehicle dynamics do matter in this case and will play a role to some extent, they are not significant enough to overturn this trend, at least not for the emission model considered by this paper. (2) When the link is controlled by a signal light, a higher link occupancy indicates more severe congestion and a longer queue. As a result more vehicle stop-and-go movements are expected to occur and contribute to the total emission amount.
\begin{figure}[h!]
\centering
\includegraphics[width=.65\textwidth]{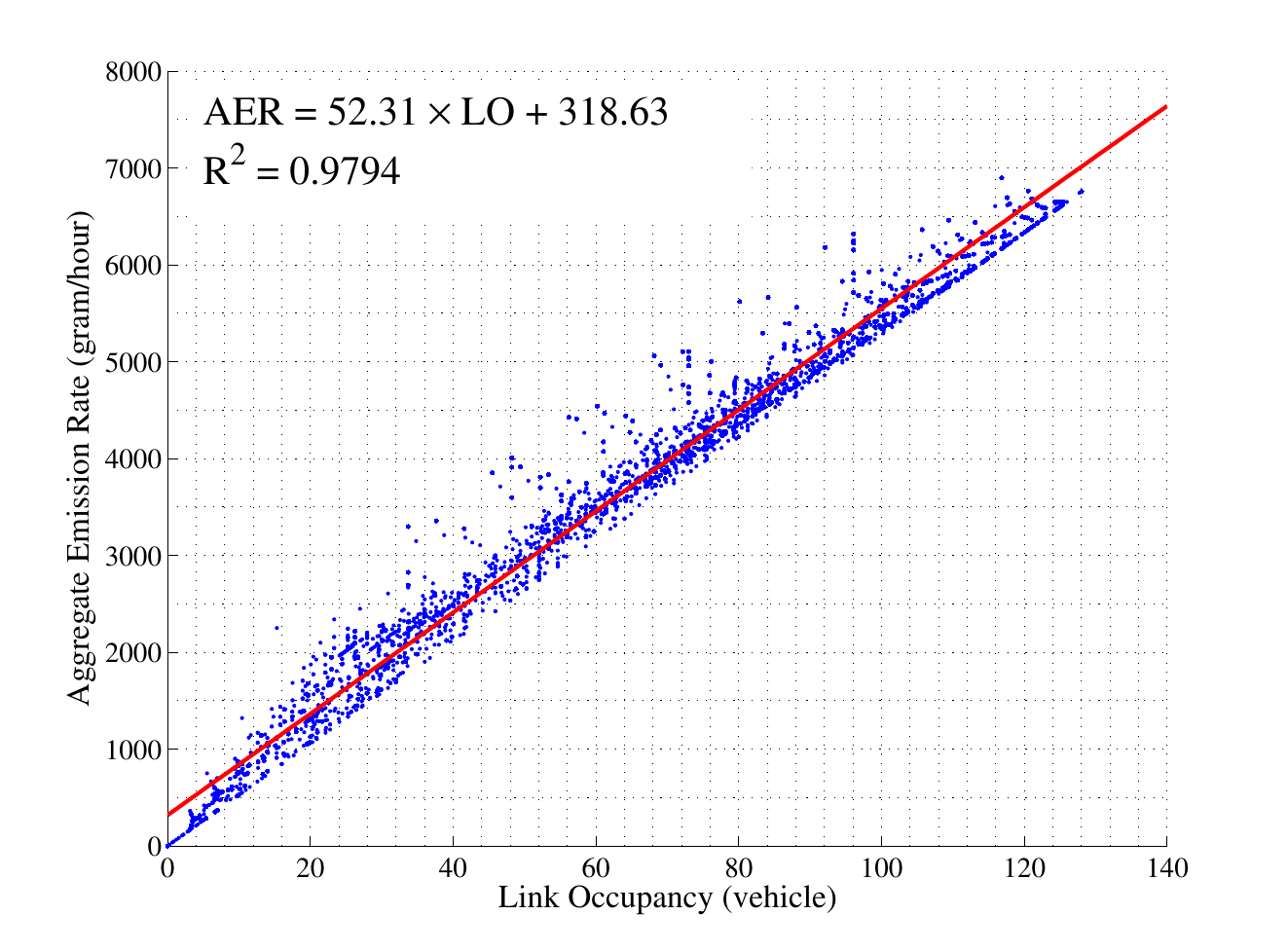}
\caption{Scatter plot of link occupancy vs. AER.}
\label{figModal_emission}
\end{figure}

The observed relationship between LO and AER may be approximated in a number of ways, e.g. using linear, piecewise linear, or piecewise smooth curve fittings. In Figure \ref{figModal_emission} we show the fitting result using linear approximation, with $R^2=0.9794$. We note, however, that this does not mean that the linear fit is the best choice; and other function forms may be more appropriate. However, there is a trade-off between the sophistication/goodness of the curve fitting and the simplicity and computational tractability of the resulting optimization formulation, as we subsequently show in Section \ref{secRO}. Such a trade-off needs to be taken into account when one formulates the uncertainty set based on curve fitting.

Some of the significant deviations from the linear fit, as shown in the figure, are caused by the highly nonlinear vehicle dynamics (i.e. acceleration and deceleration) predicted by the LWR model under the control of signal light. More specifically, close to a signalized intersection the shock waves and vehicle stop-and-go movements influenced by the signal timings generate vehicle acceleration/deceleration profiles that significantly contribute to the emission of pollutants. Moreover, the emission amount are related to the actual spatial configuration of vehicle densities and speeds on the link, which cannot be adequately captured by the link occupancy alone.

In order to show that similar macroscopic relationship exists for some other choices of link parameters, say length, we show in Figure \ref{figtwoemission} two additional simulation results for $L=200$ meters and $L=800$ meters (the result shown in Figure \ref{figModal_emission} is based on $L=400$ meters). Although similar macroscopic relationships exist in these two additional cases, the uncertainty set needs to be calibrated separately.

\begin{figure}[h!]
\centering
\begin{minipage}[c]{0.49\textwidth}
\centering
\includegraphics[width=\textwidth]{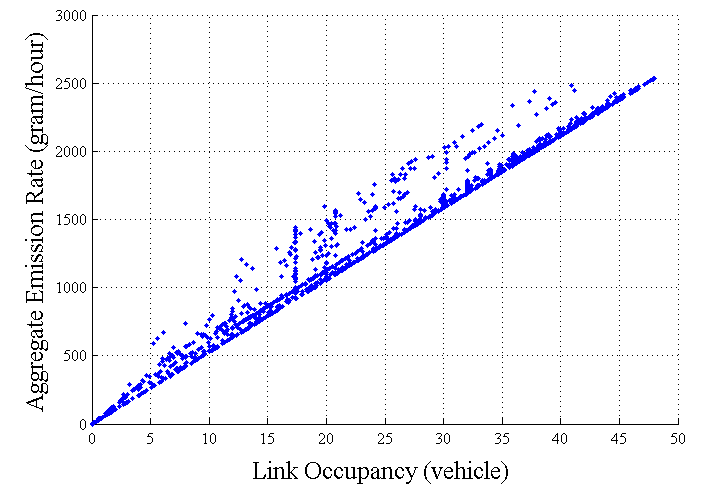}
\end{minipage}
\begin{minipage}[c]{0.49\textwidth}
\centering
\includegraphics[width=\textwidth]{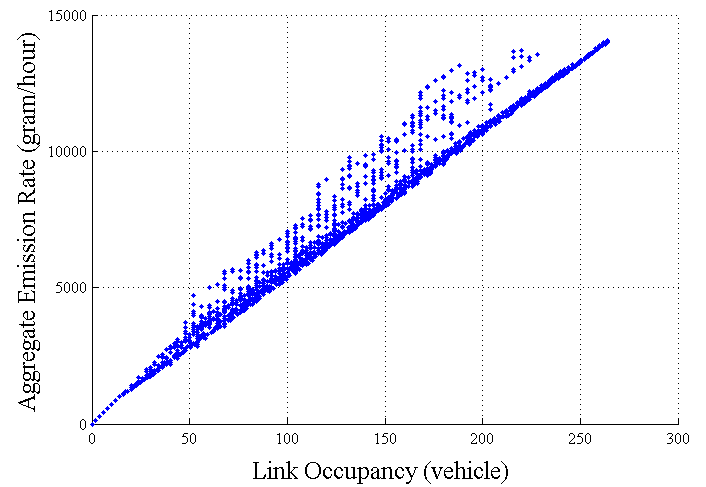}
\end{minipage}
\caption{Scatter plots of link occupancy vs. AER. Left: $L=200$ meters; right: $L=800$ meters.}
\label{figtwoemission}
\end{figure}

The macroscopic relationships depicted in Figure \ref{figModal_emission} and Figure \ref{figtwoemission} are quite interesting from a robust optimization perspective. The reasons is that although the points are spread out  in certain region, they are sparse in some places while dense in some other regions. An uncertainty region that simply covers all these points may be too conservative; and a more effective approach requires a more refined and sophisticated calibration of the uncertainty set. Detailed calibration of the uncertainty set associated with Figure \ref{figModal_emission} will be presented in Section \ref{subsecnummacro}.

\section{Signal control with emission constraints}\label{secRO}

This section provides a general mathematical framework for incorporating emission-related side constraints, while relying on fairly general assumptions on the macroscopic relationship between the aggregate emission rate and the link occupancy. We also present explicit reformulations of the emission side constraints for three special cases, namely, when the relationship is (1) affine; (2) convex piecewise affine; and (3) concave piecewise affine. Notably, all the three reformulations lead to linear constraints, which, when combined with the signal control formulation presented earlier, do not alter the nature of the mixed integer linear program.

We propose a set of emission-related  constraints constructed in a data-driven manner, which employs robust optimization techniques to handle prediction errors arising from the macroscopic relationship. In addition, we show that when this relationship is either affine or convex/concave piecewise affine, the resulting emission-related  side constraints are still linear.

For an arbitrary link in the network, we let $N(t)$ be the occupancy of this link at time $t$. Here, for notation convenience the subscript `$i$' indicating the link ID is dropped in this section.  The time-varying aggregate emission rate (AER) on this link is denoted $AER(t)$. In the following derivation of emission constraints we assume a general polynomial form relating $N(t)$ to $AER(t)$.  As a result,  the proposed framework can handle a wide range of  macroscopic relationships.

We assume the macroscopic relationship is approximated by a polynomial with degree ${\mathcal L}$, where the approximation is based on regression analysis or other types of curve-fitting techniques:
\begin{equation}\label{polyrel}
AER\big(N(t);\,\boldsymbol a\big)~\doteq~ \sum_{l=0}^{\mathcal L} a_ l(N(t))^ l~=~\boldsymbol a^T \boldsymbol N(t)\qquad  \forall t\in[0,\,T]
\end{equation}
\noindent Here $\boldsymbol  a\doteq (a_0,\,a_1,\,\ldots,\,a_{\mathcal{L}})^T\in\mathbb{R}^{\mathcal{L}}$ is a vector of coefficients in the polynomial, and $\boldsymbol N(t)\doteq \left((N(t))^0,\,\ldots,\,(N(t))^{\mathcal{L}}\right)^T$ is a vector-valued function of time. With  these notations, we consider the first type of side constraint, which stipulates that the total emission amount at the subject link is constrained by a prescribed level $E$.
 \begin{equation}\label{deterministic constraint} 
 \int_0^TAER\big(N(t);\,\boldsymbol a \big)\,dt~\leq~ E\qquad\forall t\in[0,\,T]
 \end{equation}

Notice that, for now, we are simply replacing the macroscopic relationship with a polynomial function without any consideration of errors associated with this  approximation. Later in the next subsection, we will take into account approximation errors and handle them using robust optimization techniques.

Side constraints of the form \eqref{deterministic constraint} can be easily generalized to address other types of environmental considerations including the following.
\begin{itemize}
\item The total emission on a subset of the network is bounded by some given value.
\item The total emission on a subset of the network is minimized.
\item  The differences among the total emissions on a subset of links are bounded and/or minimized.
\end{itemize}
\noindent The first case is a trivial extension of the constraint \eqref{deterministic constraint}, and will not be elaborated in this paper. The second case can be easily handled by invoking the ``epigraph reformulation''. The third case ensures that no link (or nearby environment) suffers much more than other parts of the network in terms of pollutant emission. This issue is identified by \cite{BR} as {\it environmental equity}. Detailed treatment of the last two environmental considerations will be presented in Section \ref{secgeneralization}.

\subsection{Emissions side constraints: A general formulation based on robust optimization}\label{subsecesc}
Inequality (\ref{deterministic constraint}) involves a polynomial approximation of the relationship between the link occupancy and the aggregate emission rate. In order to ensure that the emission constraint is still satisfied in the presence of approximation errors, we consider the following constraint instead:
 \begin{equation}\label{robust constraint} 
 \int_0^TAER\big(N(t);\,\boldsymbol{a}(t)\big)\,dt~\leq~ E\qquad\forall \boldsymbol a(\cdot)\in{\boldsymbol\eta}_{a}
 \end{equation}
\noindent  where the coefficients in the polynomial are allowed to vary over time; that is, $\boldsymbol a(t)=(a_ l(t):0\leq  l\leq {\mathcal L})$, and 
 \begin{equation}\label{most abstract}
AER\big(N(t);\,\boldsymbol a(t)\big)~=~\sum_{l=0}^{\mathcal L} a_ l(t)(N(t))^ l
\end{equation}
 
 \noindent Moreover, ${\boldsymbol\eta}_{a}$ is specified as a budget-like uncertainty set, which is similar to what is proposed by \cite{muhong}:
\begin{equation}
{\boldsymbol\eta}_{a}~=~\left\{\boldsymbol a(\cdot):~L_{l}\leq a_{l}(t)\leq U_{ l},~\forall t\in[0,\,T],~\forall 0\leq  l\leq {\mathcal L}; ~\sum_{l=1}^{{\mathcal L}}\int_{0}^{T}{a_ l(t)}\,dt~\leq~ \frac{T\sum_{ l=1}^{\mathcal L} U_{ l}}{\sigma}\right\} \label{uncertainty set}
 \end{equation}
where $L_ l$ and $U_ l$, $0\leq  l\leq {\mathcal L}$, are respectively the lower and upper bounds of the corresponding coefficient; and $\sigma$, which is used to adjust the conservatism of the robust optimization, satisfies
\begin{equation}\label{sigmabound}
\sigma\in\left[1,~~{{\sum_{l=0}^{\mathcal L}U_l}  \over {\sum_{l=0}^{\mathcal L}L_l}} \right]
\end{equation}

\noindent In \eqref{sigmabound}, the upper bound on \(\sigma\) ensures that the uncertainty set expressed in \eqref{uncertainty set} is nonempty. This can be easily seen by manipulating the last inequality of \eqref{uncertainty set}:
$$
\sum_{l=1}^{\mathcal{L}}T\cdot L_l~\leq~\sum_{l=1}^{\mathcal{L}}\int_{0}^Ta_l(t)\,dt~\leq~{T\sum_{l=1}^{\mathcal{L}}U_l\over \sigma}
$$
Thus $\sigma$ must not exceed ${\sum_{l=1}^{\mathcal L}U_l}/{\sum_{l=1}^{\mathcal L}L_l}$.  The constraints described by \eqref{robust constraint}-\eqref{sigmabound} correspond to the notion of robust optimization in the sense that the emission amount on the link of interest is bounded from above with any possible realization of the parameter $\boldsymbol a(t)$.

\begin{remark}
By writing \eqref{uncertainty set} we have implicitly assumed that coefficients associated with the zeroth-order term are not correlated with the other coefficients (the summation starts from $l=1$ instead of $l=0$). This implicit assumption will result in a more risk-averse formulation. However, this assumption is easy to relax, and the consequent generalization of our formulation is a trivial extension.
\end{remark}

 The uncertainty set in \eqref{uncertainty set} consists of two types of constraints: (1) a box constraint prescribing the upper and lower bounds of the uncertain parameters; and (2) a constraint that stipulates the sums of uncertain coefficients to be bounded from above (the last constraint in \eqref{uncertainty set}). With just the first type of constraints, the RO will generate the most conservative solution by predicting that all the uncertain parameters are realized at the extreme case against the decision maker. However, such a worst case occurs only with a very low probability in a realistic system, and this conservative solution is most likely to compromise the performance of the resulting system. Use of the second type of constraint can reduce the conservatism by excluding some of the extreme and rare cases. 

The second type of constraint is made flexible by adjusting the value of $\sigma$. Specifically,  a higher value of $\sigma$ implies a smaller uncertainty set, which results  in solutions that are more risk-prone (less conservative). In the most conservative case, i.e. $\sigma=1$,  the last constraint in (\ref{uncertainty set}) is out of effect.   \cite{Bandi} and \cite{Bertsimas et al 2014} provide data-driven approaches based on probability theory and statistical tests to determine the parameterization of the uncertainty sets in accordance with the observed data. Those approaches provide theoretical guarantee for the  satisfaction  of constraints in a probabilistic sense.  We would like to further remark that the second type of constraint may also capture potential correlations among $a_l(t)$, $0\leq l\leq\mathcal{L}$. (In contrast, the current uncertainty set concerns only correlations among $a_l(t)$, $1\leq l\leq\mathcal{L}$.) However, due to  space limitation this aspect of research will not be elaborated in this paper.

\subsection{Time-discretization and explicit reformulation}\label{reformulation subsection}
Here, we adopt the same notation convention as in Section \ref{seclinkdt} by using superscript `$k$' to indicate association with the $k$-th discrete time step, $1\leq k\leq M$. The constraint \eqref{robust constraint} can be time-discretized into the following form, where $\delta t$ denotes the time step size. 
\begin{equation}\label{robust discrete}
\sum_{k=1}^M \sum_{ l=0}^{\mathcal L} a_{l}^k(N^{k})^ l\delta t~\leq~ E \qquad\forall \hat{\boldsymbol a}\in\hat {\boldsymbol\eta}_{a}
\end{equation}
where $a_{l}^k$ corresponds to $a_l(t_k)$ at the $k$-th time step, and  $\hat{\boldsymbol a}\doteq (a_{ l}^k: ~0\leq  l\leq {\mathcal L},~1\leq k\leq M)$. The discrete-time version of the  uncertainty set is
\begin{equation}
\hat {\boldsymbol\eta}_{a}=\left\{\hat{\boldsymbol a}:L_ l\leq a_{ l}^k\leq U_ l, ~\forall 1\leq k\leq M,~\forall 0\leq  l\leq {\mathcal L};~\sum_{k=1}^M\sum_{ l=1}^{\mathcal L} a_l^k\delta t\leq\frac{T\sum_{l=1}^{\mathcal L}U_ l}{\sigma}\right\}
\end{equation}
The constraint \eqref{robust discrete} is in fact a semi-infinite constraint with an infinite index set $\hat {\boldsymbol\eta}_{a}$, which means that the inequality needs to be satisfied for infinitely many $\hat{\boldsymbol a}$. This makes it not directly computable. The following theorem provides a computable reformulation of \eqref{robust discrete} using dual variables.
\begin{theorem}{\bf (Robust reformulation with polynomial relationship)}\label{theorem 1}
Let constants $L_l<U_l,\, 0\leq l\leq \mathcal L$ and $\sigma$ be given. If \(\hat{\boldsymbol\eta}_{a}\) has nonempty interior, the semi-infinite constraint (\ref{robust discrete}) is equivalent to the following set of constraints:
\begin{gather}
\sum_{ l=1}^{\mathcal L}\sum_{k=1}^M \beta_{l}^kU_{ l}-\sum_{l=1}^{\mathcal L}\sum_{k=1}^M \gamma_{ l}^kL_{ l}+\frac{M\sum^{\mathcal L}_{l=1}U_ l}{\sigma}\,\theta+ M U_{0}\,\delta t~\leq~ E\label{robustre1}
\\ 
s.t.\quad \beta_{l}^k-\gamma_{l}^k+\theta~=~(N^{k})^ l\delta t\qquad\forall~ 1\leq  l\leq {\mathcal L},  \qquad 1\leq k\leq M   \label{robustre2}
\\ 
\beta_{l}^k,~\gamma_{ l}^k,~\theta~\geq~ 0\qquad\forall~ 1\leq  l\leq {\mathcal L}, \qquad 1\leq k\leq M   \label{robustre3}
\end{gather}
where \(\beta_{l}^k,~\gamma_{l}^k\) and \(\theta\) are dummy variables, $1\leq l\leq\mathcal{L}$, $1\leq k\leq M$.
\end{theorem}

\begin{proof} 
The constraint \eqref{robust discrete} can be trivially rewritten as:
\begin{equation}
\max_{\hat{\boldsymbol a}\in\hat{\boldsymbol\eta}_{a}}\sum_{k=1}^M \sum_{ l=1}^{\mathcal L} a_{ l}^k (N^{k})^ l\delta t+
\max_{\hat{\boldsymbol a}\in\hat{\boldsymbol\eta}_{a}}\sum_{k=1}^M  a_{0}^k(N^{k})^0\delta t~\leq~ E,
\end{equation}
which is equivalent to
\begin{equation}\label{robust equivalent}
\max_{\hat{\boldsymbol a}\in\hat{\boldsymbol\eta}_{a}}\sum_{k=1}^M \sum_{ l=1}^{\mathcal L} a_{ l}^k(N^{k})^ l\delta t~\leq~ 
E - \sum_{k=1}^M  U_0\delta t 
\end{equation}
The left hand side of \eqref{robust equivalent} involves solving a parametric problem of the form:
\begin{align}
\label{prime1}
&\max \sum_{k=1}^M \sum_{ l=1}^{\mathcal L} a_{ l}^k(N^{k})^ l\delta t
\\
\label{prime2 const 1}
s.t.~~& a_{ l}^k~\leq~ U_ l \quad \forall 0\leq  l\leq {\mathcal L},~1\leq k\leq M 
\\
\label{prime2 const 2}
& a_{ l}^k~\geq~ L_ l \quad \forall 0\leq  l\leq {\mathcal L},~1\leq k\leq M
\\
\label{prime2}
&\sum_{k=1}^M \sum_{ l=1}^{\mathcal L} a_{ l}^k\delta t~\leq~ \frac{T\sum_{ l=1}^{\mathcal L}U_ l}{\sigma}
\end{align}
\noindent where we treat each $N^{k}$ as a constant parameter, $1\leq k\leq M$. The program \eqref{prime1}-\eqref{prime2} corresponds to the following dual problem: 
\begin{gather}
\min \sum_{ l=1}^{\mathcal L}\sum_{k=1}^M \beta_{l}^k U_{ l}-\sum_{ l=1}^{\mathcal L}\sum_{k=1}^M \gamma_{ l}^k L_{ l}+\frac{T\sum^{\mathcal L}_{ l=1}U_ l}{\sigma\delta t}\theta     \label{dual1}
\\ 
s.t.\quad  \beta_{ l}^k-\gamma_{ l}^k+\theta~=~(N^{k})^ l\delta t\qquad\forall ~1\leq  l\leq {\mathcal L},\qquad 1\leq k\leq M      \label{dual2}
\\
 \beta_{ l}^k,~\gamma_{ l}^k,~\theta~\geq~ 0\qquad\forall~ 1\leq  l\leq {\mathcal L},\qquad 1\leq k\leq M   \label{dual3}
\end{gather}
\noindent where $\beta_{l}^k$, $\gamma_{l}^k$ and $\theta$ are dual variables corresponding to constraints \eqref{prime2 const 1}, \eqref{prime2 const 2} and \eqref{prime2}, respectively. Under the assumption that \(\hat{\boldsymbol\eta}_{a}\) has nonempty interior, by noticing the compactness of \(\hat{\boldsymbol\eta}_{a}\), the primal program \eqref{prime1}-\eqref{prime2} and the dual program \eqref{dual1}-\eqref{dual3} have finite solutions and zero duality gap. Moreover, by duality, the objective value of any feasible solution of the dual problem \eqref{dual1}-\eqref{dual3} provides an upper bound of the primal problem \eqref{prime1}-\eqref{prime2}. Therefore, if there exist $\beta_{l}^k$, $\gamma_{l}^k$ and $\theta$ such that \eqref{dual2}-\eqref{dual3} are satisfied, then if $N^{k}$ satisfies \eqref{robustre1}, \eqref{robust equivalent} is also satisfied.

On the other hand, if there exists $N^{k}$ that satisfies \eqref{robust equivalent}, then the objective value of the optimal solution to the parametric problem \eqref{prime1}-\eqref{prime2} is bounded above and thus there exists $\beta_{l}^k$, $\gamma_{l}^k$ and $\theta$ such that \eqref{dual2}-\eqref{dual3} are satisfied. This shows the desired equivalence result.
\end{proof}

\begin{remark}
 Theorem \ref{theorem 1} is based on the discussions in \cite{bn1999} and \cite{Bertsimas etal2011a}. With this reformulation, the original semi-infinite constraint is now computable with standard nonlinear programming techniques. We will call the reformulation with constraints \eqref{robustre1}-\eqref{robustre3} the robust counterpart of the original robust problem \eqref{robust discrete}.
\end{remark}

The nonlinearity of constraints (\ref{robustre1})-(\ref{robustre3}) is caused by the powers of $N^{k}$ appearing in the right hand side of \eqref{robustre2}. And they obviously stem from the polynomial approximation \eqref{most abstract} with degree greater than one.  It is not difficult to see that linearity will be retained if $\mathcal{L}=1$, i.e. when the macroscopic  relationship is approximated with an affine function. This observation leads to the simplified reformulation discussed below.

\subsection{A special case when the macroscopic relationship is affine}\label{subsecaffine}
In a special case where the relationship between the link occupancy (LO) and the aggregate emission rate (AER) is approximately affine, the robust reformulation discussed previously is considerably simplified. Indeed, we can reduce \eqref{robust constraint} to
\begin{equation}\label{original linear}
\int_{0}^{T} \big[a_1(t) N(t) + a_0(t)\big]\,dt~\leq~ E\qquad\forall \big(a_0(\cdot),a_1(\cdot)\big)\in {\boldsymbol\eta}_{a}
\end{equation}
\noindent where the set ${\boldsymbol\eta}_{a}$ reduces to
\begin{equation}\label{uncertainty set affine}
{\boldsymbol\eta}_{a}~=~\left\{\big(a_0(\cdot),\,a_1(\cdot)\big):~L_{ l}\leq a_{ l}(t)\leq U_{ l},~\forall t,~l=0,\,1; ~\int_{0}^{T}{a_1(t)}dt\leq \frac{T U_{1}}{\sigma}=\frac{M\delta t U_{1}}{\sigma}\right\} 
 \end{equation}
 
\noindent With time-discretization, we have the following formulation:
\begin{equation}\label{lubounds}
\delta t\sum_{k=1}^M a_{1}^kN^k+a_{0}^k ~\leq~ E  \qquad\forall (a_{0},~a_{1})\in\hat {\boldsymbol\eta}_{a}
\end{equation}
\noindent where $a_0\doteq (a_{0}^k:\, 1\leq k\leq M)$, $a_1\doteq (a_{1}^k:\, 1\leq k\leq M)$. The uncertainty set is given as
\begin{equation}\label{affuncset}
\hat{\boldsymbol\eta}_{a}=\left\{(a_{0},\, a_{1}):L_0\leq a_{0}^k\leq U_0,~ L_1\leq a_{1}^k\leq U_1,~\forall 1\leq k\leq M; ~\sum_{k=1}^{M} a_{1}^k\leq \frac{M U_{1}}{\sigma}\right\}
\end{equation}
With this uncertainty set, we have the following result concerning the robust reformulation, which is a special case of Theorem \ref{theorem 1}.

\begin{corollary}{\bf (Robust reformulation with affine relationship)}\label{coraffine} If $\hat{\boldsymbol\eta}_{a}$ has nonempty interior, then the semi-infinite constraint \eqref{lubounds} is equivalent to the following set of constraints.
\begin{gather}\label{roconstraints1}
\sum_{k=1}^{M} U_{1} \beta^{k}-\sum_{k=1}^{M} L_{1} \gamma^{k}+\frac{M U_{1}}{\sigma}\theta+ MU_{0}\,\delta t~\leq~ E
\\ 
s.t.\quad \beta^{k}-\gamma^{k}+\theta~=~\delta t N^k \qquad \forall 1\leq k\leq M \label{roconstraints2}
\\
\beta^{k},~\gamma^{k},~\theta~\geq~ 0\qquad\forall 1\leq k\leq M \label{roconstraints3}
\end{gather}
where $\beta^{k}$,  $\gamma^{k}$,  $\theta$ are dual variables. 
\end{corollary}
Once the semi-infinite constraint \eqref{lubounds} is reformulated as a finite set of linear constraints according to Corollary \ref{coraffine}, the LWR-E problem with an affine macroscopic relationship can be formulated and solved as a mixed integer linear program. Such an MILP will be numerically tested in Section \ref{secNE}.

\subsection{Piecewise affine macroscopic relationship}
This section formulates the emission constraints for a class of more general macroscopic relationships, namely, piecewise affine functions. Since any continuous function can be approximated by a piecewise affine function, the formulation provided below will accommodate a wider range of macroscopic relationships not yet presented in this paper. In order to retain linearity in the constraints, we consider two types of piecewise affine relationships: convex piecewise affine and concave piecewise affine. Given the same objective to be minimized, convexity and concavity in the macroscopic relationship render completely different structures of the robust optimization and the MILP reformulation.  More specifically, the convex piecewise affine constraint essentially means that \underline{all} affine pieces can be bounded from above by a prescribed level in the worst case scenario. In contrast, the concave piecewise affine constraint means that \underline{at least one} affine piece can be bounded from above in the worst case scenario. Such a difference looms large, especially, when we consider a computable reformulation in the form of a set of mixed integer linear constraints.  In the following, we will discuss the convex case in Subsection \ref{secconvex piecewise linear} and the concave case in Subsection \ref{secconcave piecewise linear}.

\subsubsection{Convex piecewise affine relationship}\label{secconvex piecewise linear}
It will be seen in this section that the previously presented reformulation can be easily extended to treat the convex piecewise affine case. We denote by \(\mathbb M\) the index set of affine pieces that constitute the piecewise affine function. For each affine piece $m\in\mathbb{M}$, we let  $b_{0,m}(t)$ and $b_{1,m}(t)$ be the zeroth- and first-order coefficients, respectively. We further employ the notation ${\boldsymbol b}(t)\doteq \big(b_{\tau,m}(t):~\tau\in\{0,\,1\},~m\in\mathbb M\big)$ for  \(t\in[0,~T]\).

With these notations, the convex piecewise affine approximation of the AER is expressed as:
\begin{equation}
AER\big(N(t);\,{\boldsymbol b}(t)\big)~\doteq~ \max_{m\in \mathbb M}\Big\{b_{1,m}(t)\cdot N(t)+b_{0,m}(t)\Big\}\qquad t\in[0,~T]\label{piecewise linear}
\end{equation}
\noindent Again, we consider the following emission constraint:
\begin{equation}
\int_{0}^TAER\big(N(t);\,{\boldsymbol b}(t)\big)\,dt~\leq~ E\qquad \forall \boldsymbol b(\cdot)\in{\boldsymbol\eta}_{b}\label{piecewise continuous}
\end{equation}
\noindent where the budget-like uncertainty set is 
\begin{multline}
{\boldsymbol\eta}_{b}=\left\{\vphantom{\sum_{m\in\mathbb M}\int_{0}^{T}{b_{1,m}(t)}}\boldsymbol b(\cdot):~L_{\tau,m}\leq b_{\tau,m}(t)\leq U_{ \tau,m},~\forall\tau\in\{0,\,1\},~\forall t\in[0,\,T],~\forall m\in\mathbb M; \right.
\\\left. ~\sum_{m\in\mathbb M}\int_{0}^{T}{b_{1,m}(t)}\,dt~\leq~ \frac{T\sum_{m\in\mathbb M}U_{1,m}}{\sigma}\right\}\label{uncertainty set piecewise affine convex}
\end{multline}
\noindent The constraint \eqref{piecewise continuous} can be immediately discretized as:
\begin{equation}\label{discretePWASI}
\delta t \sum_{k=1}^M  \max_{m\in \mathbb M}\left\{ b_{1,m}^k \cdot N^{k}+b_{0,m}^k\right\}~\leq~ E \qquad \forall \hat{\boldsymbol b}\in\hat{\boldsymbol\eta}_{b} 
\end{equation}
\noindent where $\hat{\boldsymbol b}\doteq (b_{\tau,m}^k:~\tau\in\{0,1\},~1\leq k\leq M,~m\in\mathbb M)$ is the discrete-time version of ${\boldsymbol b} (\cdot)$, and the time-discretized uncertainty set is:
\begin{multline}
\hat{\boldsymbol\eta}_{b}=\left\{\vphantom{\sum_{m\in\mathbb M}\frac{T\sum_{m\in\mathbb M}U_{1,m}}{\sigma}}\hat{\boldsymbol b}:~L_{\tau,m}\leq b_{\tau,m}^k\leq U_{ \tau,m},~\forall\tau\in\{0,1\},~\forall 1\leq k\leq M, ~\forall m\in\mathbb M; \right.
\\
\left. ~\sum_{m\in\mathbb M}\delta t \sum_{k=1}^M b_{1,m}^k~\leq~ \frac{T\sum_{m\in\mathbb M}U_{1,m}}{\sigma}\right\}\label{uncertainty set piecewise affine discrete}
\end{multline}
\noindent Evidently, \eqref{discretePWASI} is equivalent to 
\begin{equation}\label{piecewise continuous new}
\max_{\hat{\boldsymbol b}\in\hat{\boldsymbol\eta}_{b} }\delta t \sum_{k=1}^M  \max_{m\in \mathbb M}\left\{ b_{1,m}^k \cdot N^{k}+b_{0,m}^k\right\}~\leq~ E
\end{equation}
Even though the  piecewise affine function is convex, the robust counterpart \eqref{piecewise continuous new} involves a convex maximization (or equivalently, concave minimization) problem, which is nonconvex and intractable in general. The following theorem shows that the robust constraint \eqref{piecewise continuous new} can be equivalently rewritten as a set of linear constraints with the aid of dual variables. 

\begin{theorem}{\bf (Robust reformulation with convex piecewise affine relationship)}\label{piecewise linear convex theorem}
Let constants $L_{\tau, m}<U_{\tau,m}$ and $\sigma$ be given, $\tau\in\{0,\,1\},\, m\in\mathbb{M}$. Define a set of matrices 
$$
\mathbb V\doteq\left\{\left(\upsilon_{m,k}:~m\in\mathbb M,~1\leq k\leq M\right)\in \{0,\,1\}^{\vert \mathbb M\vert\times M}:~\sum_{m\in\mathbb M}\upsilon_{m,k}=1,~\forall\, 1\leq k\leq M~\right\}
$$
\noindent If \(\hat{\boldsymbol\eta}_{b}\) has nonempty interior, then the semi-infinite constraint (\ref{discretePWASI}) is equivalent to the following finite set of linear constraints: 
\begin{align}\label{piecewise 1}
\begin{rcases}
\\&\displaystyle  \sum_{{m}\in\mathbb M}\sum_{{k}=1}^M \left(-L_{1,{m}}\gamma_{1,{m}}^{s,k}+\beta_{1,{m}}^{s,k} U_{1,{m}}+U_{0,{m}}\beta_{0,{m}}^{s,k}-L_0\gamma_{0,{m}}^{s,k}\right)+\frac{M\sum_{{m}\in\mathbb M}U_{1,{m}}}{\sigma}{\theta^{s}}~\leq~ E \nonumber
\\
&-\gamma_{1,{m}}^{s,k}+\beta_{1,{m}}^{s,k}+{\theta^{s}}~=~ N^{k} \upsilon^{s}_{m,k}\delta t\qquad \forall 1\leq {k}\leq M,~{m}\in\mathbb M
\\
&-\gamma_{0,{m}}^{s,k}+\beta_{0,{m}}^{s,k}~=~\upsilon^{s}_{m,k}\delta t\qquad \forall 1\leq {k}\leq M,~{m}\in\mathbb M 
\\ 
& \gamma_{1,{m}}^{s,k}~\geq~ 0,~\beta_{1,{m}}^{s,k}~\geq~ 0, ~\gamma_{0,{m}}^{s,k}~\geq~ 0,~\beta_{0,{m}}^{s,k}~\geq~ 0\qquad \forall  1\leq {k}\leq M,~{m}\in\mathbb M 
\\
& {\theta^{s}}~\geq~ 0
\end{rcases}~\forall \upsilon_{m,k}^{s}\in\mathbb V
\end{align}
\noindent where $\beta_{0,m}^{s,k},~\gamma_{0,m}^{s,k}$, $\beta_{1,m}^{s,k},~\gamma_{1,m}^{s,k}$ and $\theta^{s}$ are dual variables.
\end{theorem}
\begin{proof} 
The proof is moved to \ref{appendix A} for the conciseness of our presentation.
\end{proof}

Similar to Corollary \ref{coraffine} which handles the affine relationship, Theorem \ref{piecewise linear convex theorem} provides a computable formulation that preserves linearity in the constraints even though the relationship between AER and LO is in a more general functional form. Notice that, however, the set of constraints shown in Theorem \ref{piecewise linear convex theorem} have to be satisfied for each and every matrix $\upsilon_{m,k}^s$ in $\mathbb{V}$. In other words, there are $|\mathbb{V}|$ copies of such set of constraints, and each copy is parameterized by $s$. The cardinality  $\vert \mathbb V\vert$ of the finite set $\mathbb V$ is $\vert \mathbb M\vert^M$. This means that the number of linear constraints grows exponentially. Nonetheless, these exponentially many  constraints are all linear and do not involve any integer variables. We also make note of the fact that the presence of exponentially many constraints is not rare in combinatorial optimization problems; examples include the traveling salesman problem and the minimum spanning tree problem. In addition, algorithms such as the cutting plane method may be employed to systematically enumerate these constraints, which may lead to reduction in the computational cost. \cite[For an example of this type of algorithms, see][]{Nemhauser and Sigismondi}.

\subsubsection{Concave piecewise affine relationship}\label{secconcave piecewise linear}

We show in this section that, when the relationship is piecewise affine and concave, one can maintain the linearity in the constraints by modifying the uncertainty sets and introducing additional integer variables. With the same notations introduced in Section \ref{secconvex piecewise linear}, we define the concave and piecewise affine  relationship as
\begin{equation}\label{concave piecewise linear}
AER\big(N(t);\, {\boldsymbol b}(t)\big)~\doteq~ \min_{m\in \mathbb M}\Big\{b_{1,m}(t)\cdot N(t)+b_{0,m}(t)\Big\}\qquad t\in[0,~T] 
\end{equation}
\noindent Notice that the ``max" operator from the convex case has now been changed to ``min". The emission constraint reads:
\begin{equation}
\int_{0}^TAER\big( N(t);\,{\boldsymbol b}(t)\big)\,dt~\leq~ E \qquad \forall \boldsymbol b(\cdot)\in{\boldsymbol\eta}_{b} \label{concave piecewise continuous}
\end{equation}
Here we have modified the uncertainty set \({\boldsymbol\eta}_{b}\) to be a product of sets: ${\boldsymbol\eta}_b=\eta_{b,1}\times\eta_{b,2}\times\dots\times \eta_{b,m}\times\dots\times\eta_{b,\vert\mathbb M\vert}$ with 
\begin{equation}
\eta_{b,m}=\left\{\boldsymbol b_m(\cdot):~L_{\tau,m}\leq b_{\tau,m}(t)\leq U_{ \tau,m},~\forall\tau\in\{0,1\},~\forall t\in[0,\,T]; ~\int_{0}^{T}{b_{1,m}(t)}\,dt~\leq~ \frac{TU_{1,m}}{\sigma_m}\right\} \label{uncertainty set piecewise affine}
\end{equation}
where \(\boldsymbol b_m(\cdot)\doteq (b_{\tau,m}(\cdot):~\tau\in\{0,1\})\).
Unlike the uncertainty set defined in \eqref{uncertainty set piecewise affine convex}, in this case the uncertain parameters associated with different affine pieces are uncorrelated. In other words, the constraints related to the uncertain parameters of different affine pieces in the uncertainty set are decoupled.  
A time-discretization of constraint \eqref{concave piecewise continuous} is given as following:
\begin{equation}\label{concave discrete piecewise linear}
\sum_{k=1}^M \min_{m\in \mathbb M}(  b_{1,m}^k\cdot N^{k}+b_{0,m}^k)\delta t~\leq~ E
\qquad 
\forall \hat{\boldsymbol b}\doteq (\hat{\boldsymbol b}_1,\dots,\hat{\boldsymbol b}_{\vert\mathbb M\vert})\in\hat{\boldsymbol \eta}_b\doteq \hat{\eta}_{b,1}\times\dots\times\hat{\eta}_{b,\vert\mathbb M\vert}
\end{equation}
\noindent where $\hat{\boldsymbol b}_m\doteq (b_{\tau,m}^k:~\tau\in\{0,1\},~1\leq k\leq M)$,  $\forall  m\in\mathbb M$. The uncertainty set $\hat{\eta}_{b,m}$ is defined as
\begin{equation}
\hat{\eta}_{b,m}=\left\{\hat{\boldsymbol b}_m:~L_{\tau,m}\leq b_{\tau,m}^k\leq U_{ \tau,m},~\forall\tau\in\{0,1\},~\forall 1\leq k\leq M,~\sum_{k=1}^M{b_{1,m}^k}\delta t~\leq~ \frac{TU_{1,m}}{\sigma_m}\right\}\quad m\in\mathbb{M}
\end{equation}
The following theorem presents a reformulation of the robust constraint as a set of mixed integer linear constraints.
\begin{theorem}{\bf (Robust reformulation with concave piecewise affine relationship)}\label{thmconcavepwa}
Let constants  $L_{\tau,m}<U_{\tau,m}$ and $\sigma_m$ be given for all $\tau\in\{0,1\}$ and $m\in\mathbb M$. If \(\hat{\boldsymbol\eta}_{b}\) has nonempty interior, the semi-infinite constraint (\ref{concave discrete piecewise linear}) is equivalent to the following set of linear constraints:
\begin{align}\label{concave piecewise 1}
\begin{split}
&\sum_{m\in\mathbb M}\left\{-\sum_{k=1}^M L_{1,m}\gamma_{1,m}^k+\sum_{k=1}^M\beta_{1,m}^k U_{1,m}+\theta_m\frac{M U_{1,m}}{\sigma_m}+\sum_{k=1}^M U_{0,m}\beta_{0,m}^k-\sum_{k=1}^ML_0\gamma_{0,m}^k\right\}~\leq~ E
\\
&N^k \delta t- (1-\upsilon_{m,k})\mathcal M\leq -\gamma_{1,m}^k+\beta_{1,m}^k+\theta_m ~\leq~ (1-\upsilon_{m,k})\mathcal M+N^k \delta t \qquad \forall  1\leq k\leq M,~m\in\mathbb M
\\
&-\upsilon_{m,k}\mathcal M\leq -\gamma_{1,m}^k+\beta_{1,m}^k+\theta_m ~\leq~ \upsilon_{m,k}\mathcal M  \qquad \forall  1\leq k\leq M,~m\in\mathbb M
\\
&
-\gamma_{0,m}^k+\beta_{0,m}^k~=~\upsilon_{m,k}\delta t \qquad \forall 1\leq k~\leq~ M,~m\in\mathbb M
\\ 
& \gamma_{1,m}^k\geq 0,~\beta_{1,m}^k~\geq~ 0, ~\gamma_{0,m}^k\geq 0,~\beta_{0,m}^k\geq 0\qquad \forall  1\leq k\leq M,~m\in\mathbb M
\\
& \upsilon_{m,k}\in\{0,~1\}\qquad \forall  1~\leq~ k\leq M,~m\in\mathbb M
\\
& \theta_m~\geq~ 0\qquad~m\in\mathbb M
\\
& \sum_{m\in\mathbb M}\upsilon_{m,k}~=~1\qquad \forall  1\leq k\leq M
\end{split}
\end{align}
\noindent where $\beta_{1,m}^k,~\gamma_{1,m}^k\), \(\beta_{0,m}^k,~\gamma_{0,m}^k$ and $\theta_{m}$ are dual variables and $\mathcal M>0$ is  sufficiently large.
\end{theorem}
\begin{proof} 
The proof is postponed until \ref{appendix B}.
\end{proof}

The concave piecewise affine relationship results in  nonconvex constraints, which can bring substantial computational challenges. In fact, these nonconvex constraints cannot be well handled by existing commercial solvers, including CPLEX and Gurobi.  Nonetheless, with Theorem \ref{thmconcavepwa}, we can again retain linearity by introducing additional integer variables. We further note that due to the last constraint in \eqref{concave piecewise 1}, these integer variables are highly correlated so that the search space of the problem grows only on an order of $M$. Therefore, our reformulation significantly reduces the computational ramifications of handling a set of nonconvex constraints.

\section{Generalization of the robust optimization approach}\label{secgeneralization}
As mentioned right before Section \ref{subsecesc}, the proposed robust optimization formulation can be easily extended to capture two additional types of environmental considerations: (1)  to minimize the total emission on a subset of the network; and (2) to minimize the differences in the total emission on some relevant links, or,  to ``equalize" the emissions among the links (environmental equity). These two generalizations will be discussed in this section.

\subsection{Minimizing emissions}\label{subseceminimize}
The proposed model is easily generalizable to the case where the traffic-driven emissions are subject to minimization rather than side constraints. In this case, we conveniently invoke the ``epigraph reformulation''. For example, if the objective is to minimize the total emission amount on some subset of links $\mathbb{I}_S\subset\mathbb{I}$ where $\mathbb{I}$ denotes the set of links in the network; that is, 
$$
\min \sum_{I_i\in \mathbb{I}_S} \int_0^T AER_i\big(N_i(t);\,\boldsymbol{a}_i(t)\big)\,dt
$$
\noindent We may simply minimize a dummy variable $z$ with one additional constraint  
$$
\sum_{I_i\in\mathbb{I}_S} \int_{0}^T AER_i\big(N_i(t);\,\boldsymbol a(t)\big)\,dt~\leq~ z
$$
This additional constraint reduces to the form of \eqref{deterministic constraint}, and our previously presented results apply here.

\subsection{Environmental equity}\label{equity}

The equity constraint is relevant when the network planner needs to enforce emissions to be distributed relatively evenly across all or part of a network. To address this consideration, we consider the constraint of the following form:
\begin{equation}
\int_{0}^TAER_i\big(N_i(t);\,\boldsymbol a_i(t)\big)\,dt - \int_{0}^T AER_j\big(N_j(t);\,\boldsymbol a_j(t)\big)\,dt~\leq~ E_{ij}\qquad \forall I_i,~I_j\in\mathbb{I}\nonumber
\end{equation}
\noindent where $E_{ij}$ is some prescribed threshold. We assume for now no uncertainty in the parameters $\boldsymbol a_i(\cdot)$ and $\boldsymbol a_j(\cdot)$. By assuming an affine relationship between the emission rate and link occupancy, without loss of generality, this constraint can be rewritten as
$$
\int_{0}^{T} \left\{a_{i,1}(t)N_i(t)+a_{i,0}(t)\right\}dt -\int_{0}^T \left\{a_{j,1}(t)N_{j}(t)+a_{j,0}(t)\right\}dt~\leq~ E_{ij} \qquad \forall I_i,~I_j\in\mathbb{I}
$$
Now we allow $\boldsymbol a_i(\cdot)$ and  $\boldsymbol a_j(\cdot)$ to be uncertain and the uncertainty for each  is described by budget uncertainty sets $\boldsymbol \eta_{a_i}$ and $\boldsymbol\eta_{a_j}$ similar to \eqref{uncertainty set}. Then we have the following robust constraint:
\begin{equation}\label{Eijrobust}
\begin{array}{c}
\displaystyle \int_{0}^{T} \left\{a_{i,1}(t)N_{i}(t)+a_{i,0}(t)\right\}dt -\int_{0}^T \left\{a_{j,1}(t)N_{j}(t)+a_{j,0}(t)\right\} dt~\leq~ E_{ij}
\\\\
 \forall \boldsymbol a_i(\cdot)~\doteq~(a_{i,1}(\cdot),~a_{i,0}(\cdot))\in{\boldsymbol\eta}_{a_i},\quad\boldsymbol a_j(\cdot)~\doteq~(a_{j,1}(\cdot),~a_{j,0}(\cdot))\in{\boldsymbol\eta}_{a_j}\quad \forall I_i,~I_j\in\mathbb{I}
 \end{array}
\end{equation}
This is equivalent to the following: $\forall I_i,\,I_j\in\mathbb{I}$,
$$
\sup_{\boldsymbol a_i(\cdot)\in{\boldsymbol\eta}_{a_i}}\left\{\int_{0}^{T} \left\{a_{i,1}(t)N_{i}(t)+a_{i,0}(t)\right\}dt\right\}
+
\sup_{\boldsymbol a_j(\cdot)\in{\boldsymbol\eta}_{a_j}}\left\{-\int_{0}^T \left\{a_{j,1}(t)N_{j}(t)+a_{j,0}(t)\right\}dt\right\}~\leq~ E_{ij}
$$
We proceed as before to time-discretize the problem. Explicit reformulation of the discretized constraints can be easily derived by invoking the dual formulations for each of the two maximization problems involved in the constraint. These dual formulations differ little from those discussed in Section \ref{subsecaffine} and thus are omitted here. As a result, linear constraints can be easily derived to replace the robust constraint \eqref{Eijrobust}, and the MILP formulation for the environmental equity problem follows.

Finally, one may also try to minimize the difference $E_{ij}$ between any two links $I_i$ and $I_j$, by adopting the same procedure discussed in Section \ref{subseceminimize} through the ``epigraph reformulation".

\section{Numerical Study}\label{secNE}

\subsection{Network setup}
In this section, we consider a hypothetical network consisting of four intersections, three of which are signalized; see Figure \ref{fignetwork}. The ten links in this network are assumed to have the same triangular fundamental diagram, whose parameters are shown below.
$$
v~=~40/3 ~\hbox{(meter/s)},\quad \rho^{jam}~=~0.4~ \hbox{veh/meter},\quad \rho^c~=~0.1 ~\hbox{(veh/meter)},\quad C~=~4/3~\hbox{(veh/s)}
$$
\noindent In addition, all links have the same length of $L=400$ meters.

\begin{figure}[h!]
\centering
\begin{minipage}[c]{0.41\textwidth}
\centering
\includegraphics[width=\textwidth]{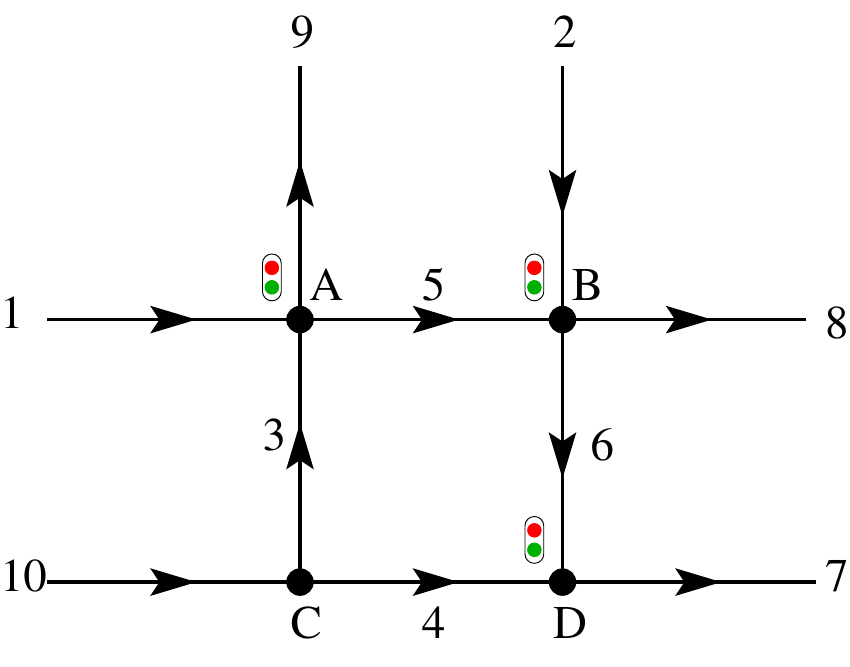}
\end{minipage}
\begin{minipage}[c]{0.57\textwidth}
\centering
\includegraphics[width=\textwidth]{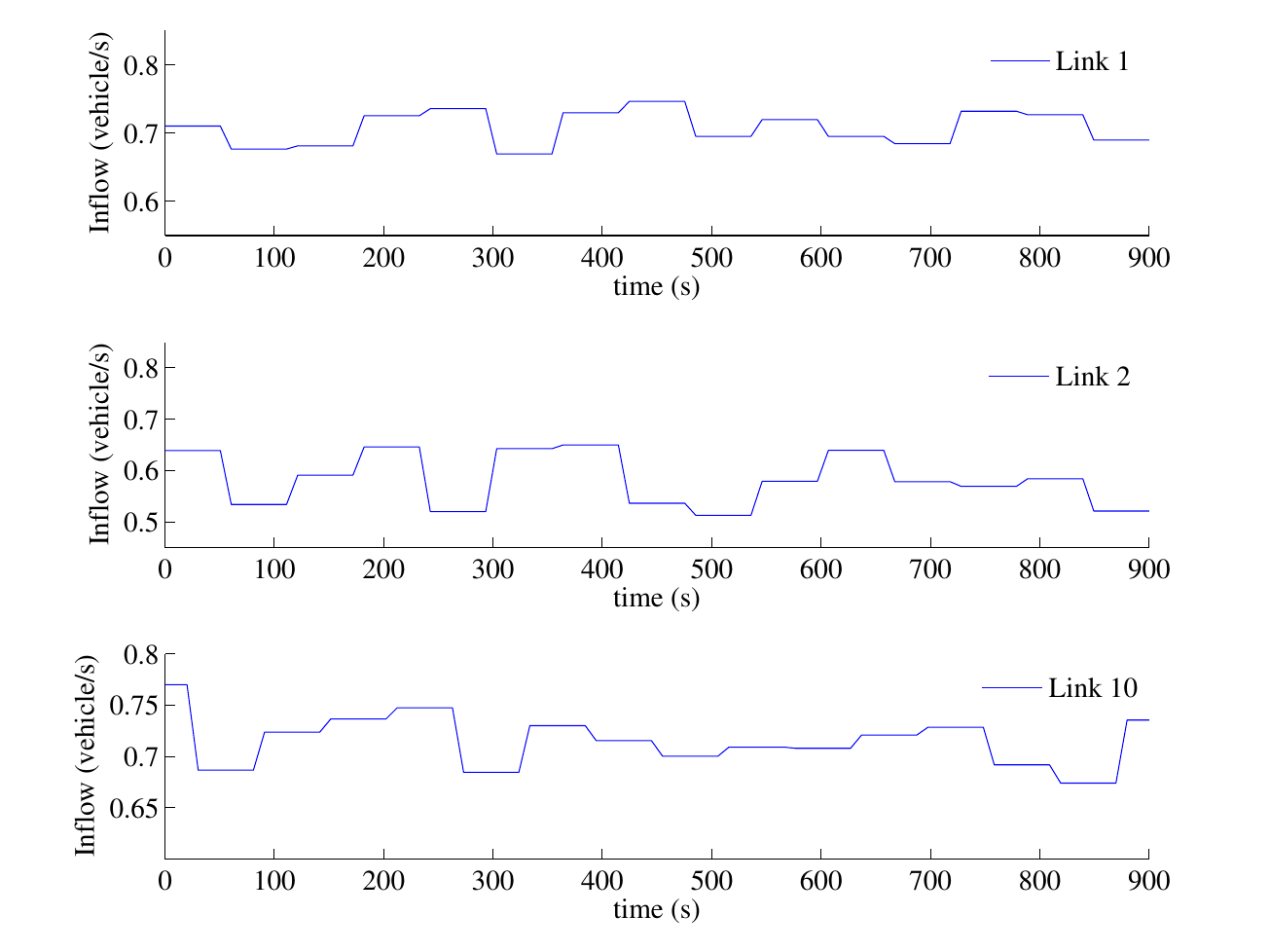}
\end{minipage}
\caption{Test network (left); and the demand profile corresponding to Scenario I (right).}
\label{fignetwork}
\end{figure}

For simplicity, we assume that routing is such that vehicles at all intersections have a fixed probability of selecting either of the two downstream approaches; and the turning ratios are specified as follows.
$$
\alpha_{1,5}~=~0.50,\qquad\alpha_{3,5}~=~0.40,\qquad \alpha_{2,6}~=~0.30,\qquad\alpha_{5,6}~=~0.50,\qquad \alpha_{10,3}~=~0.53
$$
Note that this assumption is not essential to our formulation or computation, but was made to simplify the presentation of results. In reality these turning percentages may be estimated based on turn-by-turn vehicle counts at intersections. The time horizon of our numerical example is a 15-min time period, with a time step of $10$ seconds.

In order to test the effectiveness of the proposed signal timing in reducing both congestion and emission under different levels of congestion, we consider three scenarios with three different levels of demand at the upstream ends of all boundary links.  The demand profile in the first scenario, which is shown in Figure \ref{fignetwork}, corresponds to the lightest traffic load. Table \ref{tabdemands} shows, for each of the three scenarios, the ratios between the average link inflows and the link flow capacity.

\begin{table}[h!]
\centering
\begin{tabular}{|c|c|c|c|}\hline
& Scenario I (light) &       Scenario II (medium) &  Scenario III (heavy)\\
\hline                 
Link 1   & 53.1 \%     &    60.6 \%    &   64.4 \%         \\
Link 2    & 43.7 \%   &     51.2 \%    &   54.9 \%        \\
Link 10  & 51.6 \%  &     68.4 \%       &   83.38 \%    \\
 \hline
\end{tabular}
\caption{Demand levels in the three scenarios with different traffic loads. The values in the table represent ratios between the average link inflows and the link flow capacity.}
\label{tabdemands}
\end{table}

Throughout this numerical study, the MILPs were  solved with ILOG Cplex 12.1.0, which ran with Intel Xeon X$5675$ Six-Core 3.06 GHz processor provided by the Penn State Research Computing and Cyberinfrastructure.

\subsection{Calibration of the uncertainty set}\label{subsecnummacro}

In this section, the macroscopic relationship between a link's aggregate emission rate (AER)  and its link occupancy (LO) will be analyzed to inform the uncertainty set used in the robust optimization. Since the link parameters in this numerical example are identical to those employed in Section \ref{subsecsimsetup}, we may focus on the macroscopic relationship depicted in Figure \ref{figModal_emission} without conducting further numerical experiments.

The macroscopic relationship shown in Figure \ref{figModal_emission}  is approximately affine with the following regression coefficients
\begin{equation}\label{nummr}
AER~\approx~a_1 \times LO+ a_0~=~52.31\times LO + 318.63
\end{equation}
where $AER$ (in gram/hour) denotes the aggregated emission rate of hydrocarbon on a link level; $LO$ (in number of vehicles) denotes the link occupancy. 
\begin{figure}[h!]
\centering
\includegraphics[width=.65\textwidth]{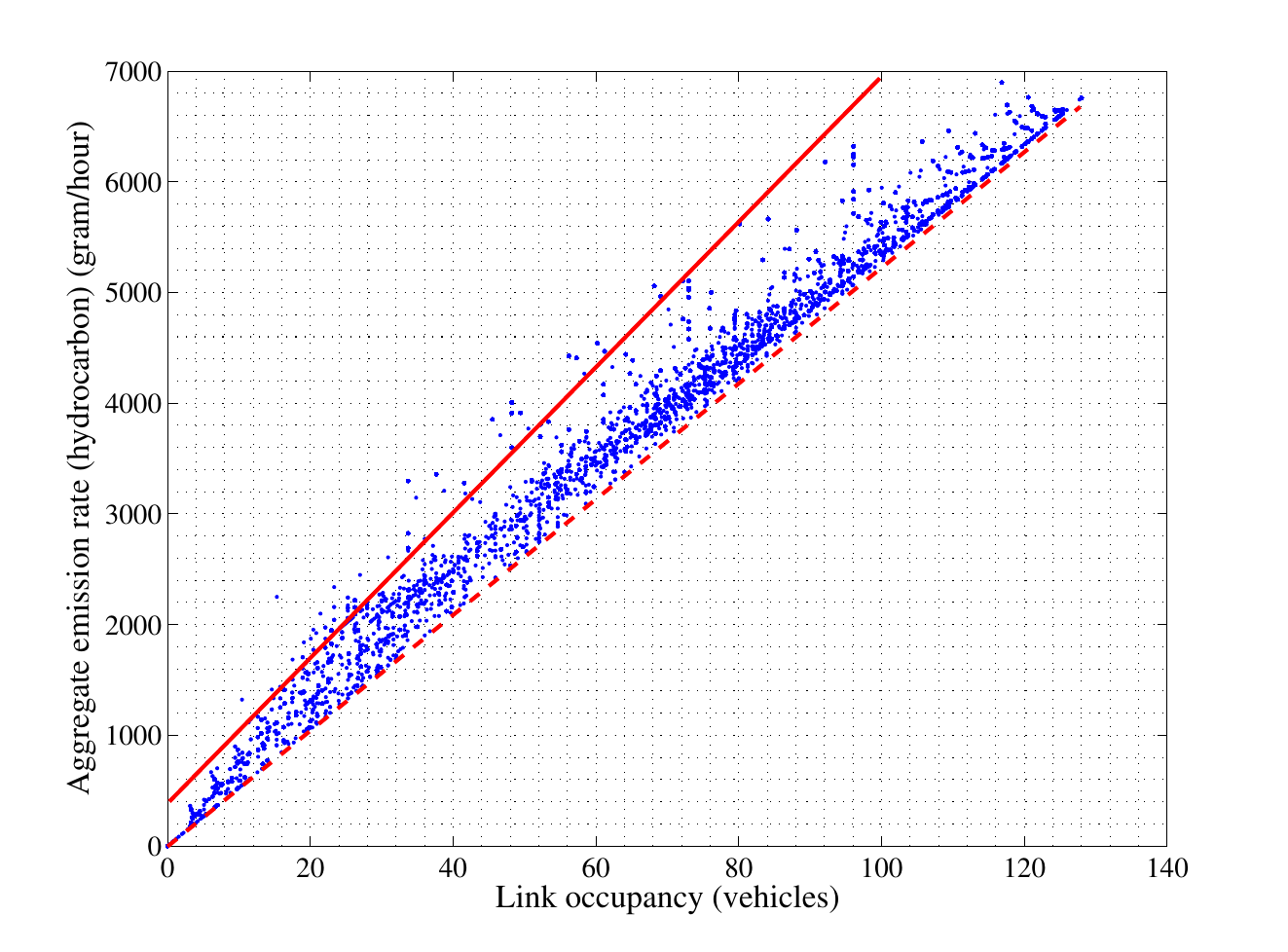}
\caption{Scatter plot of the macroscopic relationship obtained from simulation. The solid line represents the upper envelop of the uncertainty region; the dashed line represent the lower envelop of the uncertainty region.}
\label{figSimulation_Result}
\end{figure}

In order to construct the uncertainty set of the form \eqref{affuncset} for the robust optimization, we select the following lower and upper bounds for the affine coefficients:
\begin{equation}\label{numlbb1}
L_1~\leq~a_1~\leq~U_1,\qquad L_0~\leq~a_0~\leq~U_0
\end{equation}
\noindent where 
\begin{equation}\label{numlbb2}
L_0=0,\quad U_0=400,\quad L_1=53.3,\quad U_1=66
\end{equation}
\noindent The corresponding lower envelop (by letting $a_1=L_1$, $a_0=L_0$) and upper envelop  (by letting $a_1=U_1$, $a_0=U_0$) of the uncertain region are now shown in Figure \ref{figSimulation_Result}. Notice that one has a lot of freedom in choosing these upper and lower bounds; however, they do affect the performance of the resulting robust optimization. As we commented in Section \ref{subsecaffine}, the upper envelop sets a worse-case value for the emission rates, and it is likely to overestimate the emission rates for some/most scenarios. From Figure \ref{figSimulation_Result} we see that while the upper envelop provides a tight bound on the emission rates when the traffic is relatively light (i.e., $0\leq LO\leq 40$ vehicles), it tends to mostly overestimate the emission rates when the traffic volume grows (i.e., $LO\geq 40$ vehicles). To avoid the robust constraints being too conservative, we invoke the parameter $\sigma$ introduced in \eqref{affuncset} to adjust the conservativeness by allowing some realized coefficients $a_{1}^k$ to be strictly less than the upper bound $U_1$, that is,
\begin{equation}\label{nesigma}
\sum_{k=1}^M a_{1}^k~\leq~M{U_1\over \sigma}\qquad \hbox{for some }~\sigma\in\left[1,\,{U_1\over L_1}\right]
\end{equation}
where $a_{1}^k$ can be interpreted as the first-order coefficient in an actual (realized) instance of the relationship between $LO$ and $AER$ at the $k$-th time step. A general rule of thumb is that the more the upper bound seems to overestimate, the larger $\sigma$ should be, although  ideally such a choice should be specifically quantified and even optimized based on available data. Interested reader is referred to \cite{Bandi} and \cite{Bertsimas et al 2014} for some discussions on data-driven calibration of the uncertainty set. In our particular example, we choose $\sigma=1.2$. Other choices of $\sigma$ can be also considered but will not be elaborated in this paper.

Notice that the region formed by the lower and upper envelops does not contain all the points shown in the figure. However, the majority (96.06\%) of these points fall within this region, and we treat the rest as outliers. The reason for ignoring these outliers is that they only make up 3.94\% of the total dataset, and are sparsely distributed outside (and consistently above) the uncertainty region. Including these points in the uncertainty set would make our estimation too conservative by considering the (very small) chance that these emission rates occur.

The uncertainty set for the robust optimization, according to \eqref{affuncset}, is therefore constructed as follows. 
\begin{equation}\label{numuncs}
\hat\eta_{a,i}=\left\{(a_{0,k},\, a_{1,k}):0\leq a_{0,k}\leq 400,~ 53.3\leq a_{1,k}\leq 66,~\forall 1\leq k\leq M, ~\sum_{k=1}^{N} a_{1,k}\leq 55M\right\}
\end{equation}
for all link $I_i\in\mathbb{I}$ in the network, where $M$ is the number of time intervals in this problem. Notice that the same uncertainty set \eqref{numuncs} is used for all the links in the network because they have identical parameters; if not, the uncertainty sets needs to be calibrated separately.

\subsection{The base case}

For comparison purposes, we first consider a base case where the traffic signal timing is optimized, but without any emission considerations. This is achieved by simply solving the mixed integer linear program introduced in Section \ref{secLWRMILP}. In the base case we are mainly concerned with maximizing network throughput and minimizing delays; thus in view of \eqref{mipobj} we adopt the following objective function:
\begin{equation}\label{nsobj}
\max \sum_{k=1}^M {1\over 1+k} \left(\hat q^k_7+\hat q^k_{8}+\hat q^k_{9}\right)
\end{equation}
where $\hat q_7^k$, $\hat q_8^k$ and $\hat q_9^k$ are the exit flows at the $k$-th time step on the outgoing links $7$, $8$, and $9$, respectively. An objective function of the form \eqref{nsobj} tends to maximize the network throughput at any instance of time, and is similarly considered by \cite{GVM2, CSC, HGPFY}.

The base case results corresponding to the three demand levels are summarized in Table \ref{tabbase}. The table also includes the hydrocarbon emission amount on each link, which is calculated from the MILP solution and the detailed modal emission model elaborated in Section \ref{secmodal}. The purpose of this table is two-fold: (1) to enable a comparison between the base case and the emission-constrained case presented later; and (2) to suggest appropriate upper bounds on the emission amount for the emission-constrained case. In the presentation of our results, we only consider links $1$ through $6$, since the rest of the links are not directly controlled by the signals under consideration.

\begin{table}[h!]
\centering
\begin{tabular}{|c|c|c|c|c|c|c|c|}
\hline
\multicolumn{1}{|c|}{\multirow{4}{*}{Scenario I} }  &  Objective value & \multicolumn{6}{|l|}{5.331}
\\
\cline{2-8}
&  \multicolumn{1}{|c|}{\multirow{2}{*}{HC emissions (gram)}}   & link 1  & link 2   & link 3     & link 4   & link 5   & link 6  
\\
\cline{3-8}  
&  &  390.5 &  309.1  &  208.6  &  158.7  & 343.4  & 210.2
\\
\cline{2-8} 
& Total emission (gram) & \multicolumn{6}{|l|}{1620.5}

\\
\hline 
\multicolumn{1}{c}{\multirow{8}{*}{}}\\
\hline

\multicolumn{1}{|c|}{\multirow{4}{*}{Scenario II} }  &  Objective value & \multicolumn{6}{|l|}{6.615}
\\
\cline{2-8}
&  \multicolumn{1}{|c|}{\multirow{2}{*}{HC emissions (gram)}}   & link 1  & link 2   & link 3     & link 4   & link 5   & link 6  
\\
\cline{3-8}  
&  &  558.0 &  400.4  &  263.5  &  214.3  & 514.8  & 252.9
\\
\cline{2-8} 
& Total emission (gram) & \multicolumn{6}{|l|}{2203.9}

\\
\hline 
\multicolumn{1}{c}{\multirow{8}{*}{}}\\
\hline

\multicolumn{1}{|c|}{\multirow{4}{*}{Scenario III} }  &  Objective value & \multicolumn{6}{|l|}{7.142}
\\
\cline{2-8}
&  \multicolumn{1}{|c|}{\multirow{2}{*}{HC emissions (gram)}}   & link 1  & link 2   & link 3     & link 4   & link 5   & link 6  
\\
\cline{3-8}  
&  &  1359.2 &  430.6  &  560.1  &  269.1  &  553.0  & 252.7
\\
\cline{2-8} 
& Total emission (gram) & \multicolumn{6}{|l|}{3424.7}
\\\hline
\end{tabular}
\caption{Objective value and hydrocarbon (HC) emissions in the base case. The three traffic scenarios with increasing demand levels are considered. The objective value is expressed by \eqref{nsobj}.}
\label{tabbase}
\end{table}

\subsection{Simultaneous control of traffic and hydrocarbon emission}

In this subsection, we solve the signal optimization problem with emission side constraints (LWR-E). This is achieved by solving the MILP for the base case with additional emission-related robust counterpart expressed by \eqref{roconstraints1}-\eqref{roconstraints3}, where the detailed calibration of the uncertainty set is presented in Section \ref{subsecnummacro}. In view of the base case summarized in Table \ref{tabbase}, we chose the following upper bounds (in gram) on the emission amount for each link, where bounds strictly below the corresponding emissions  in the base case are underlined. 
\begin{itemize}
\item Scenario I: \quad $E_1={\underline{390}},~~ E_2=310,~~ E_3=210,~~ E_4=160,~~ E_5={\underline{310}}, ~~ E_6=240$;
\item Scenario II: \quad $E_1=600,~~ E_2={\underline{380}},~~ E_3=300,~~ E_4={\underline{210}},~~ E_5= {\underline{490}}, ~~ E_6={\underline{250}}$;
\item Scenario III: \quad $E_1={\underline{1100}},~~ E_2=440,~~ E_3=750,~~ E_4=300,~~ E_5= 600, ~~ E_6=300$.
\end{itemize}

Notice that we did not set all the bounds to be strictly below the actual emissions in the base case, because of the apparent trade-off of vehicle throughput and emission among links connected to the same intersection.  Thus, under the same network load, a signal control strategy is unlikely to simultaneously reduce the emissions on all the links. In fact, setting all the bounds to be strictly below the actual emissions in the base case is likely to yield infeasibility in most of our calculations.

\begin{table}[h!]
\centering
\begin{tabular}{|c|c|c|c|c|c|c|c|}
\hline
\multicolumn{1}{|c|}{\multirow{5}{*}{Scenario I} }  &  Objective value & \multicolumn{6}{|l|}{5.330 {\bf (0.02\% less than the base case)}}
\\
\cline{2-8}
&  \multicolumn{1}{|c|}{\multirow{2}{*}{HC emissions (gram)}}   & link 1  & link 2   & link 3     & link 4   & link 5   & link 6  
\\
\cline{3-8}  
&  \multicolumn{1}{|c|}{\multirow{2}{*}{(upper bound)}}&  390.2 &  313.3  &  204.8  &  161.3  & 289.8  & 229.0
\\
&  &  (390) &  (310)  &  (210)  &  (160)  & (310)  & (240)
\\
\cline{2-8} 
& Total emission (gram) & \multicolumn{6}{|l|}{1588.4 grams {\bf (1.98\% less than the base case)}}

\\
\hline 
\multicolumn{1}{c}{\multirow{8}{*}{}}\\
\hline

\multicolumn{1}{|c|}{\multirow{5}{*}{Scenario II} }  &  Objective value & \multicolumn{6}{|l|}{6.612 {\bf (0.05\% less than the base case)}}
\\
\cline{2-8}
&  \multicolumn{1}{|c|}{\multirow{2}{*}{HC emissions (gram)}}   & link 1  & link 2   & link 3     & link 4   & link 5   & link 6  
\\
\cline{3-8}  
&  \multicolumn{1}{|c|}{\multirow{2}{*}{(upper bound)}} &  598.5 &  369.2   &  303.8  &  214.6  & 443.8  & 247.1
\\
&  &  (600) &  (380)  &  (300)  &  (210)  & (490)  & (250)
\\
\cline{2-8} 
& Total emission (gram) & \multicolumn{6}{|l|}{2157.0 grams {\bf (2.12\% less than the base case)}}

\\
\hline 
\multicolumn{1}{c}{\multirow{8}{*}{}}\\
\hline

\multicolumn{1}{|c|}{\multirow{5}{*}{Scenario III} }  &  Objective value & \multicolumn{6}{|l|}{7.058 {\bf (1.18\% less than the base case)}}
\\
\cline{2-8}
&  \multicolumn{1}{|c|}{\multirow{2}{*}{HC emissions (gram)}}   & link 1  & link 2   & link 3     & link 4   & link 5   & link 6  
\\
\cline{3-8}  
&\multicolumn{1}{|c|}{\multirow{2}{*}{(upper bound)}}  &  1160.5 &  434.6  &  756.4  &  272.1  & 622.2  & 294.5
\\
&  &  (1100) &  (440)  &  (750)  &  (300)  & (600)  & (300)
\\
\cline{2-8} 
& Total emission (gram) & \multicolumn{6}{|l|}{3540.3 grams {\bf (3.38\% more than the base case)}}
\\\hline
\end{tabular}
\caption{Objective value and hydrocarbon (HC) emissions in the LWR-E case.}
\label{tabro}
\end{table}

The results of the proposed MILP formulation for the emission-constrained signal optimization are summarized in Table \ref{tabro}.   We see that for all the three scenarios our proposed signal optimization scheme effectively keeps the emissions below the prescribed level; and this is done at a relatively small cost to the overall throughput of the network; that is, compared with the base case, the objective values in the LWR-E case decrease by only 0.02\%, 0.05\% and 1.18\% in Scenarios I, II and III, respectively (recall that while lower total emissions are desired, the objective function is such that a lower value represents a worse performance). In addition, for Scenarios I and II, bounding the emission amount on certain links successfully reduces the total emission on the entire network, by 1.98\% and 2.12\% respectively.

In Scenario 3, where the traffic load is the heaviest, we observe that, despite a significant local emission reduction (on link 1 where emission has been reduced from 1359 grams to 1160 grams), the overall network emission increases by 3.38\%. Moreover, the total network throughput suffers more than the previous two scenarios. This further highlights the potential trade-off  not only between traffic delay and emission, but also between the local and global emission amount. And such a trade-off is likely to be significant when the network is heavily congested (demand Scenario III). Nevertheless, in certain cases a trade-off between local and global emissions, such as that shown in Scenario of Table \ref{tabro}, is desirable as some links may be closer to densely populated areas than others, and shifting the emissions on those links to other parts of the network results in a lower impact on population exposure and public health. Therefore, this numerical example highlights the need for a well-planned, multi-criteria signal optimization strategy based on well-defined {\it key performance indicators}, and additional tools for the modeling of pollutant dispersion and public exposure. These are, however, beyond the scope of this paper.

A visualization of the emission profiles for all the three test scenarios is provided in Figure \ref{figEmissions}.

\begin{figure}[h!]
\centering
\includegraphics[width=.9\textwidth]{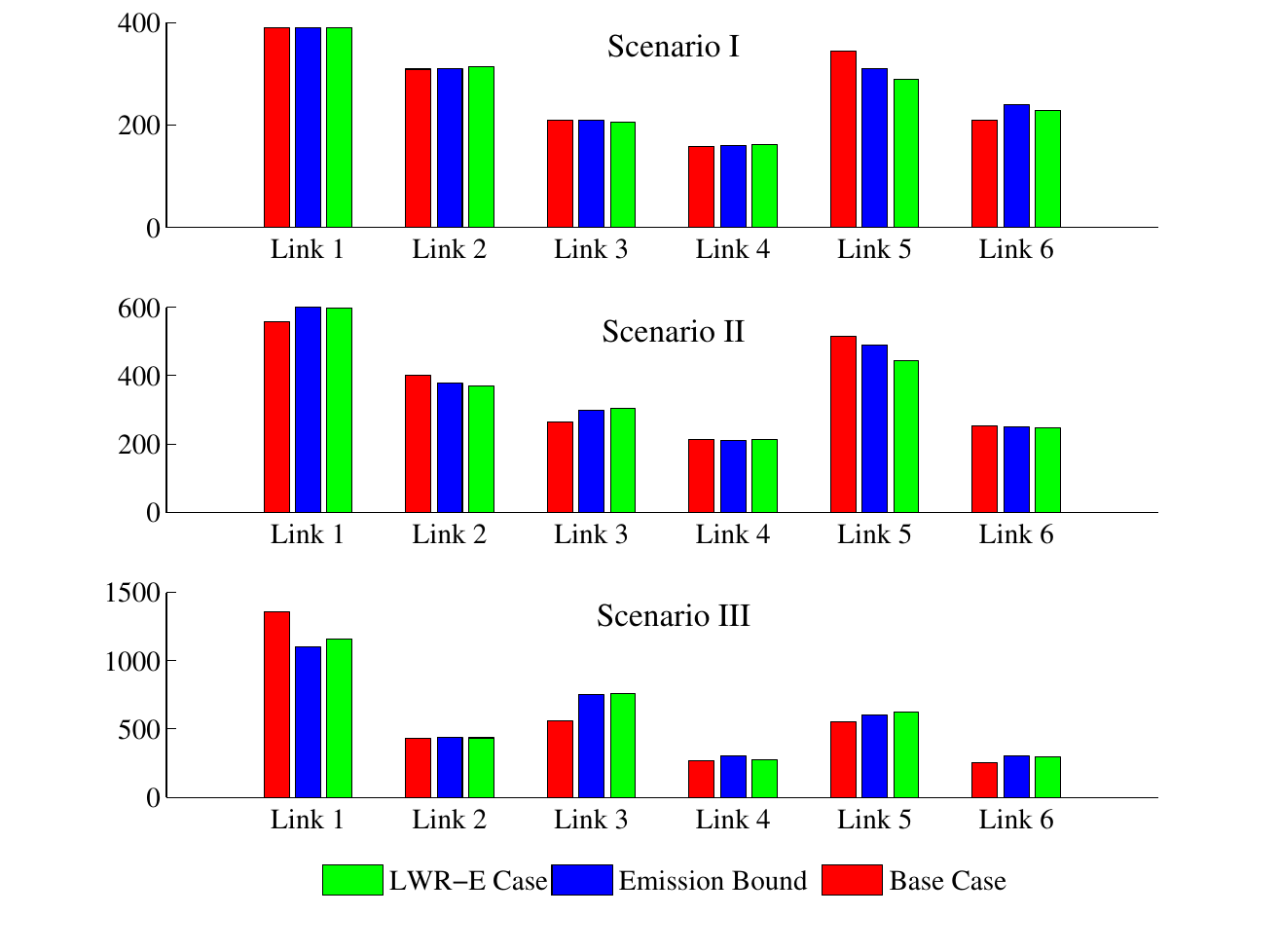}
\caption{Comparisons of link-specific HC emissions in the base cases and the emission-constrained (LWR-E) cases.}
\label{figEmissions}
\end{figure}

 We also see from Table \ref{tabro} that a few emission constraints are slightly violated; this is expected for the following two reasons: (1) the calibration process for the uncertainty set presented in Section \ref{subsecnummacro} ignores part of the data points that lie outside of the uncertainty region; (2) we chose the parameter $\sigma$ to be larger than 1 (see \eqref{nesigma}), which means that the approach taken to handle the uncertainty in the AER is relatively less conservative and could lead to, with a small probability, violation of the constraints. We present a quantification of the violations in Table \ref{tabviolation} for all the links in all the three scenarios, which shows that the violations of the constraints are within an acceptable range.

\begin{table}[h!]
\centering
\begin{tabular}{c|c|c|c|c|c|c|}
\cline{2-7}
     & link 1  & link 2   & link 3     & link 4   & link 5   & link 6  
\\
\hline  
  \multicolumn{1}{|c|}{\multirow{1}{*}{Scenario I}}   &  0.05\% &  1.06\%  &  -  &  0.81\%  & -  & -
\\
\hline
 \multicolumn{1}{|c|}{\multirow{1}{*}{Scenario II}}  &  - &  -  &  1.27\%  &  2.19\%  & -  & -
\\
\hline
 \multicolumn{1}{|c|}{\multirow{1}{*}{Scenario III}}  &  5.45\% &  -  &  0.85\%  &  -  & 3.70\%  & -
\\
\hline
\end{tabular}
 \caption{Violations of the emission constraints. ``-" means that the actual emission is below the bound.}
\label{tabviolation}
\end{table}

We investigate a specific intersection $B$ to illustrate the effects of the proposed signal control strategy. The base case is compared with the LWR-E case under Scenario II. As can be seen from Table \ref{tabbase} and Table \ref{tabro}, the total emission amount on the two incoming links, 2 and 5, are respectively 400 g and 515 g (base), and 369 g and 444 g (LWR-E). Moreover, the emission upper bounds imposed by the LWR-E problem are 380 g and 490 g. As we see from Figure \ref{figSimulation_Result} that the aggregate emission rate on a link is highly correlated to the level of congestion (i.e. the link occupancy) on that link. Thus, in order to simultaneously reduce the total emissions on links 2 and 5 to a point below their respectively upper bounds, the signal controls in the LWR-E case reduced the inflow of link 5 from the upstream node $A$ to alleviate its congestion, and the signal at node $B$ acted accordingly to reduce the congestion on link 2 by allocating more green time to that link. As a result, the congestion and emission amount on links 2 and 5 was reduced simultaneously.  However, this improvement in the LWR-E case was offset by the fact that the congestion was accumulated on links 1 and 3 as a result of reduced link 5 inflow, causing the emission on these links to increase: respectively from 558 g to 598 g (link 1) and 264 g to 304 g (link 3).

To visualize the effect of the LWR-E signal control, we show the Moskowitz functions of link 5 in the base and LWR-E cases in Figure \ref{figNcomparison} (similar trends on link 2 will not be show here due to space limitation). It can be seen that the Moskowitz surface is separated by a ``shock wave", which is shown as a kink in this case, into two domains: the uncongested region (on the side of the link entrance) and the congested region (on the side of the link exit). The separating shock wave travels back and forth as a result of the changing downstream boundary conditions caused by the signal control, indicating the growth and dissipation of queues.  From Figure \ref{figNcomparison}  we see that the LWR-E case yields less queuing on link $5$ than the base case. The comparison between the base case and the LWR-E case at junction B suggests that the LWR-E case yields less queuing near this intersection, which reduces vehicle acceleration and deceleration, both of which contribute significantly to the emissions near the stop line.

\begin{figure}[h!]
\centering
\begin{minipage}[c]{0.9\textwidth}
\centering
\includegraphics[width=.8\textwidth]{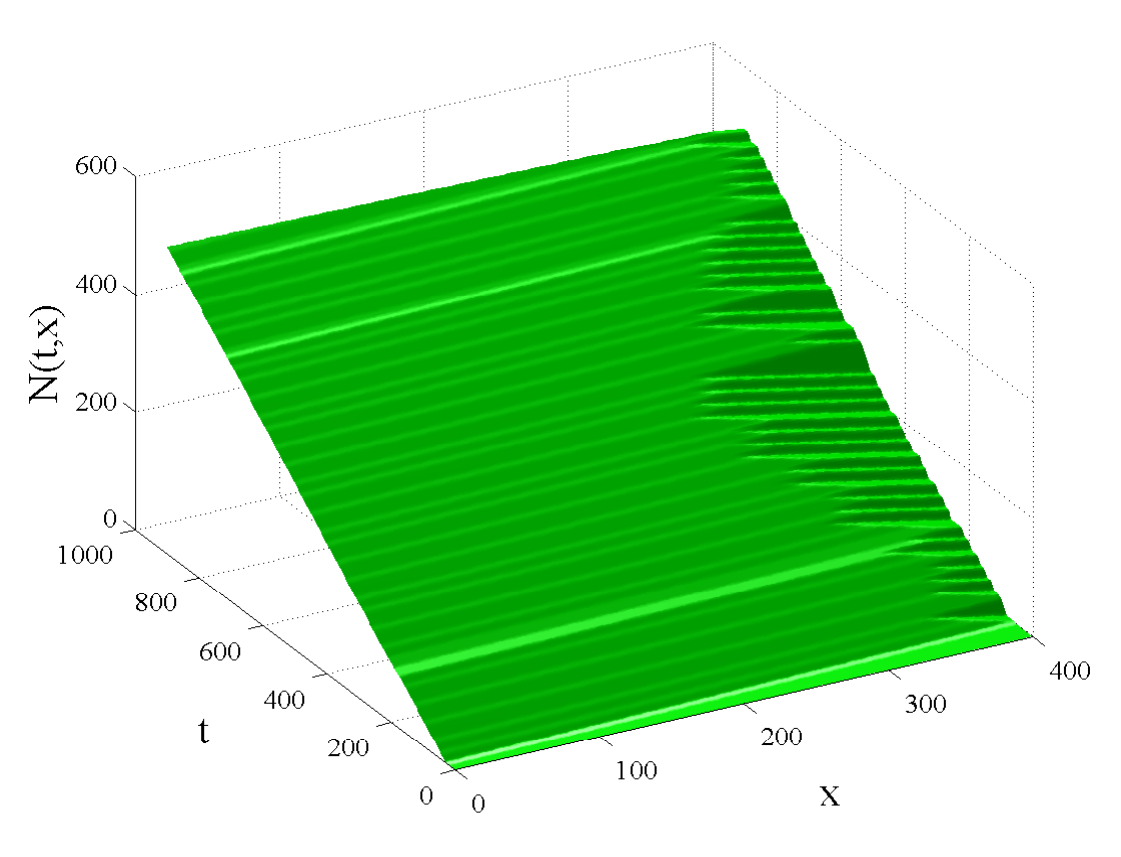}
\end{minipage}
\begin{minipage}[c]{0.9\textwidth}
\centering
\includegraphics[width=.8\textwidth]{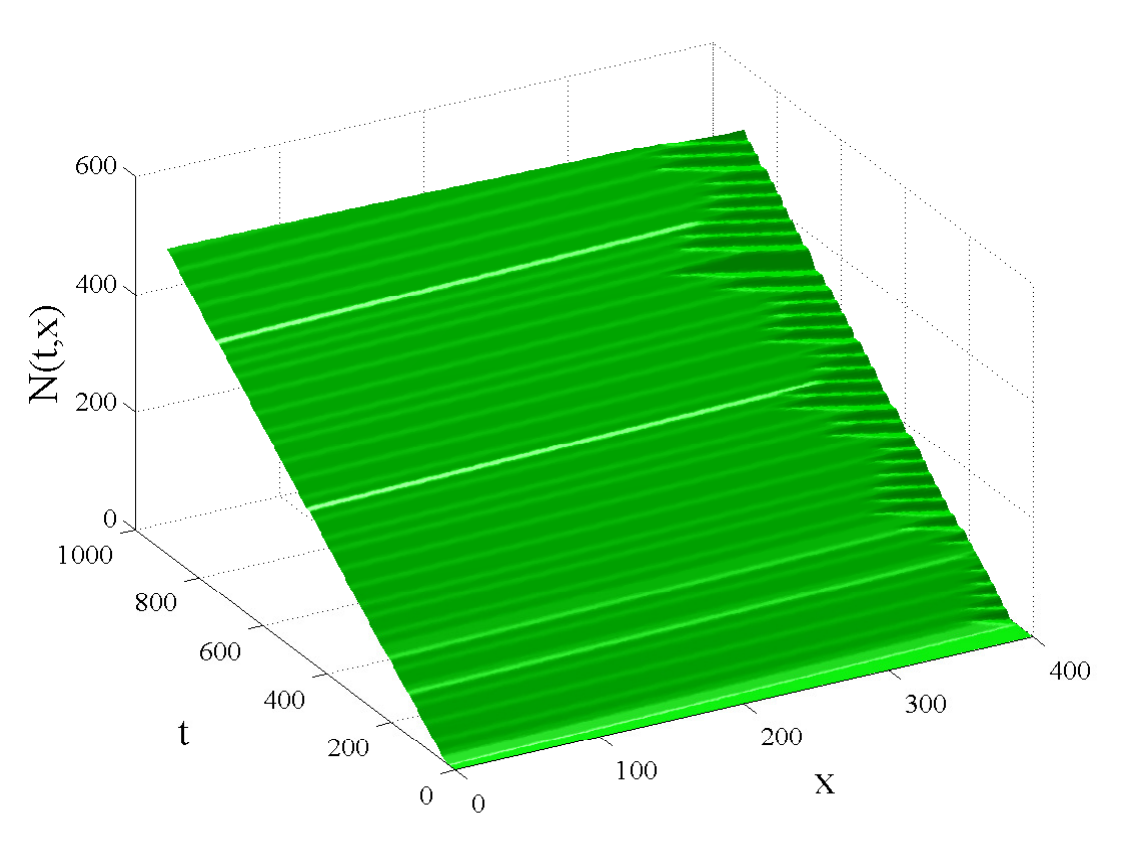}
\end{minipage}
\caption{The Moskowitz functions of link $5$ in the base case (top) and the LWR-E case (bottom).}
\label{figNcomparison}
\end{figure}

To further illustrate the effect on emission of vehicle queuing near the intersection, we show in Figure \ref{figcontour} the contour lines of the Moskowitz function of link 5, which represent vehicle trajectories in the space-time diagram.  We can clearly observe that the LWR-E case on average yields fewer vehicle stops than the base case. To further quantify this, we perform the following simple calculation: the number of stops (represented by the horizontal line segments in the trajectories) in the base case is approximately 76, while the number of stops in the LWR-E case is roughly 52. Given that there are 50 contour lines in each figure, we estimate that the average number of stops per vehicle is 1.52 in the base case, and 1.04 in the LWR-E case. Similarly, the average number of vehicle stops on link $2$ is 0.34 in the base case and 0.1 in the LWR-E case, although the corresponding contour lines are not show here. This explains the reduction of emission on these links as vehicle stops and acceleration/deceleration have been reduced by the signal controls in the LWR-E case.

\begin{figure}[h!]
\centering
\includegraphics[width=\textwidth]{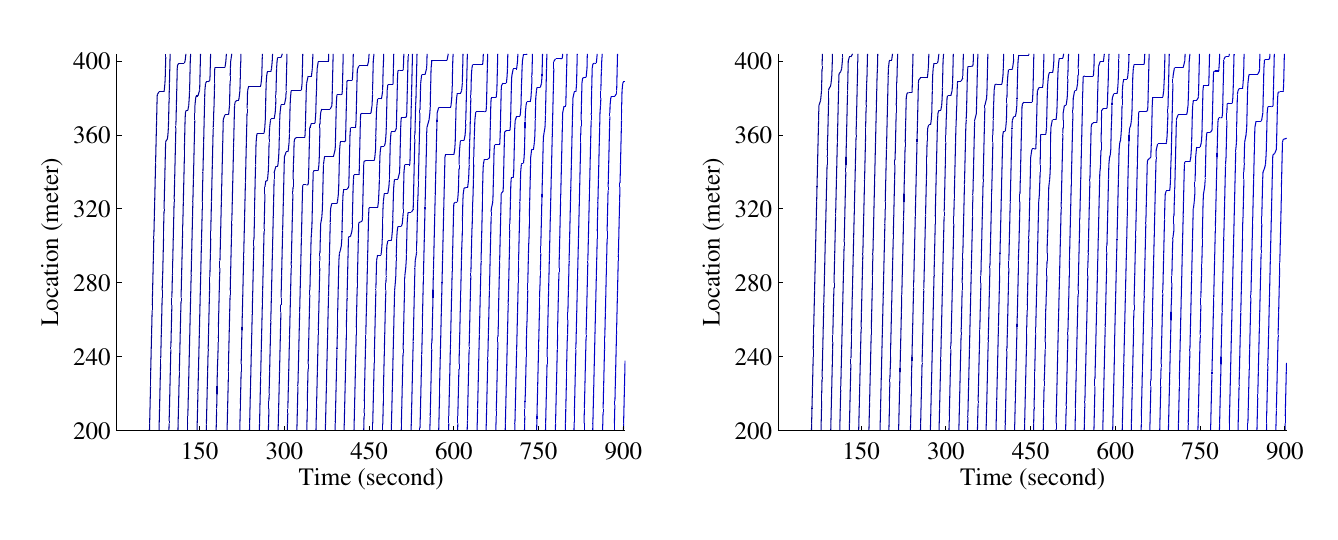}
\caption{Contour lines of the Moskowitz functions of link 5, in the base case (left), and the LWR-E case (right).}
\label{figcontour}
\end{figure}

\section{Conclusion}\label{secConclusion}

We propose here a signal optimization formulation to simultaneously minimize expected vehicle delays while accounting for constraints/objectives related to vehicular emissions throughout a network. The latter is incorporated by developing macroscopic relationships that describe the range of expected emissions on a link as a function of the occupancy of that link. Such reduced models allow emissions to be easily calculated in a network traffic model framework, as opposed to traditional techniques that require the trajectory of individual drivers in the network. Furthermore, we specifically account for the errors that exist as a result of the reduced and approximate emissions models through the use of robust optimization. This formulation entails the following theoretical advantages: (1) it is computationally tractable as a result of linearity maintained through our (piecewise) linear approximation to emission constraints; (2) it is robust to model inaccuracy through the use of robust optimization; and (3) it appropriately captures queue spillbacks to upstream links through a set of mixed integer linear constraints.

In contrast, without the RO, the only existing alternative to our knowledge is incorporating emission-relation constraints as highly nonlinear and nonconvex side constraints. The signal control model then becomes a mixed integer nonlinear program (MINLP), for which there are limited theories for analysis and scarce solution methods for computation. Existing solvers for MINLP are generally slow and a global optimal solution cannot be guaranteed. In contrast, our proposed approach results in a mixed integer linear program (MILP) for the same problem, which admits effective and well-developed solution schemes that guarantee global optimality. Moreover, state-of-the-art algorithms implemented within commercial solvers highly exploits the linear structures of MILPs and thus enjoy much improved and even close-to-tractable computational efficiency \citep{BCR}.

A numerical study on a hypothetical network demonstrates the effectiveness of the proposed procedure in bounding/reducing emissions, while maintaining a satisfactory performance of the signalized network in terms of throughputs. However, as pointed out by Scenario III in the study, the local bounding/minimization of emissions may lead to network-wide degradation in terms of both air quality and throughput. Thus, a well-balanced signal timing plan with environmental concerns requires multiple criteria and proper weighting of different objectives.

The following data would be needed to effectively apply this methodological framework in practice: knowledge of network geometries (e.g., link lengths and fundamental diagrams), measures of traffic flows at the boundaries of the network at regular intervals (perhaps through fixed-location induction loop detectors commonly available in urban networks), anticipated or historical turning fractions at all intersections (e.g. using turn-by-turn vehicle counts), and operation constraints related to signal phases, green/red times, and signal offsets. An immediate follow-up of this research will focus on an extension of the framework to account for more realistic signal control scenarios (e.g., fixed-cycle-dynamic-split, dynamic-cycle-fixed-split, etc.) through the introduction of additional constraints on signal timing parameters. The proposed linearization of the emission constraints is also applicable for on-line signal controls based on real-time information by following the general framework proposed by \cite{LHGFY}.

\appendix
\section{Proof of Theorem \ref{piecewise linear convex theorem}}\label{appendix A}

\begin{proof}
Notice that \eqref{discretePWASI} can be immediately rewritten as
\begin{equation}
\max_{\hat{\boldsymbol b}\in\hat{\boldsymbol\eta}_{b}} \delta t\sum_{k=1}^M \max_{m\in \mathbb M}\left(  b_{1,m}^k\cdot N^{k}+b_{0,m}^k\right)~\leq~ E,\nonumber
\end{equation}
\noindent which is equivalent to 
\begin{equation}\label{first step}
\max_{\hat{\boldsymbol b}\in\hat{\boldsymbol\eta}_{b},(\upsilon_{m,k})\in \mathbb{V}}\delta t\sum_{k=1}^M\sum_{m\in\mathbb M}\upsilon_{m,k}\left(  b_{1,m}^k\cdot N^{k}+b_{0,m}^k\right)~\leq~ E
\end{equation}
Here, $\mathbb{V}=\left\{(\upsilon_{m,k}:~m\in\mathbb M, ~1\leq k\leq M)\in\{0,~1\}^{\vert \mathbb M\vert\times M},~\sum_{m\in\mathbb M} \upsilon_{m,k}=1,~\forall 1\leq k\leq M\right\}$.
Consider the maximization problem for any given $\left(\upsilon_{m,k}\right)\in\mathbb{V}$ and feasible vector $\{N^{k}:~1\leq k\leq M\}$:
$$
\max_{\hat{\boldsymbol b}\in\hat{\boldsymbol\eta}_{b}}\sum_{k=1}^M\sum_{m\in\mathbb M}\upsilon_{m,k}\left(  b_{1,m}^k\cdot N^{k}+b_{0,m}^k\right)\delta t
$$
\noindent It is easy to check that the dual problem is formulated as follows:
\begin{align*}
\begin{split}
\min ~&\sum_{m\in\mathbb M}\sum_{k=1}^M \left( -L_{1,m}\gamma_{1,m}^k+\beta_{1,m}^k U_{1,m}+ U_{0,m}\beta_{0,m}^k-L_0\gamma_{0,m}^k\right) +\theta\frac{M\sum_{m\in\mathbb M}U_{1,m}}{\sigma}
\\
s.t.~~&-\gamma_{1,m}^k+\beta_{1,m}^k+\theta~=~ N^k \upsilon_{m,k}\delta t\qquad \forall 1\leq k\leq M,~m\in\mathbb M 
\\
&-\gamma_{0,m}^k+\beta_{0,m}^k~=~\upsilon_{m,k}\delta t\qquad \forall 1\leq k\leq M,~m\in\mathbb M 
\\ 
& \gamma_{1,m}^k~\geq~ 0,~\beta_{1,m}^k~\geq~ 0, ~\gamma_{0,m}^k~\geq~ 0,~\beta_{0,m}^k~\geq~ 0\qquad \forall  1\leq k\leq M,~m\in\mathbb M 
\\& \theta~\geq~ 0
\end{split}
\end{align*}

Therefore, the satisfaction of \eqref{first step} can be restated as follows:
\begin{align*}
\begin{split}
& \sum_{m\in\mathbb M}\sum_{k=1}^M \left(- L_{1,m}\gamma_{1,m}^k+\beta_{1,m}^k U_{1,m}+ U_{0,m}\beta_{0,m}^k-L_0\gamma_{0,m}^k\right)+\theta\frac{M\sum_{m\in\mathbb M}U_{1,m}}{\sigma} ~\leq~E
\\
&-\gamma_{1,m}^k+\beta_{1,m}^k+\theta~=~ N^k \upsilon_{m,k}\delta t\qquad \forall 1\leq k\leq M,~m\in\mathbb M 
\\&-\gamma_{0,m}^k+\beta_{0,m}^k~=~\upsilon_{m,k}\delta t\qquad \forall 1\leq k\leq M,~m\in\mathbb M 
\\ & \gamma_{1,m}^k~\geq~ 0,~\beta_{1,m}^k~\geq~ 0, ~\gamma_{0,m}^k~\geq~ 0,~\beta_{0,m}^k~\geq~ 0\qquad \forall  1\leq k\leq M,~m\in\mathbb M 
\\& \theta~\geq~ 0 
\end{split}
\end{align*}
\noindent This should hold for all possible choices of $\left(\upsilon_{m,k}\right)\in\mathbb V$, which immediately leads to the desired result.
\end{proof}

\section{Proof of Theorem \ref{thmconcavepwa}}\label{appendix B}
\begin{proof}
Constraint \eqref{concave discrete piecewise linear} is equivalent to stipulating that
\begin{equation}\label{some interm problem}
\max_{\hat{\boldsymbol b}\in\hat{\boldsymbol\eta}_b}\left\{\sum_{k=1}^M \min_{m\in \mathbb M}\left(  b_{1,m}^k \cdot N^{k}+b_{0,m}^k\right)\,\delta t\right\}=\max_{\hat{\boldsymbol b}\in\hat{\boldsymbol\eta}_b}\min_{\boldsymbol \upsilon\in V}\underbrace{\left\{\sum_{k=1}^M \sum_{m\in\mathbb M}\upsilon_{m,k}(  b_{1,m}^k\cdot N^{k}+b_{0,m}^k)\delta t\right\}}_{\mathcal F(\hat{\boldsymbol b},~\boldsymbol \upsilon)}~\leq~ E
\end{equation}
\noindent where $V\doteq \{{\boldsymbol \upsilon}=(\upsilon_{m,k}:~m\in\mathbb M, ~1\leq k\leq M)\in[0,\,1]^{\vert\mathbb M\vert\times M}:~\sum_{m\in\mathbb M} {\upsilon_{m,k}}=1,~\forall 1\leq k\leq M\}$.  Since $\mathcal F(\hat{\boldsymbol b},~\boldsymbol \upsilon)$ is convex in $\boldsymbol\upsilon$ and concave in $\hat{\boldsymbol b}$, we can switch the ``max'' and ``min'' operator to obtain a dual problem without duality gap, i.e.,
\begin{equation}\label{strong duality}
\max_{\hat{\boldsymbol b}\in\hat{\boldsymbol\eta}_b}\min_{\boldsymbol \upsilon\in V} \mathcal F(\hat{\boldsymbol b},~\boldsymbol \upsilon)~=~\min_{\boldsymbol \upsilon\in V}\max_{\hat{\boldsymbol b}\in\hat{\boldsymbol\eta}_b}\mathcal F(\hat{\boldsymbol b},~\boldsymbol \upsilon)
\end{equation}
\noindent for any feasible vector  $(N^{k}:~1\leq k\leq M)$. Combining \eqref{some interm problem}  and  \eqref{strong duality}, we have an alternative formulation of constraint \eqref{concave discrete piecewise linear} given as
\begin{align}
E~\geq ~&\max_{\hat{\boldsymbol b}\in\hat{\boldsymbol\eta}_b}\min_{\boldsymbol \upsilon\in V} \mathcal F(\hat{\boldsymbol b},~\boldsymbol \upsilon)\nonumber
\\=~&\min_{\boldsymbol \upsilon\in V}\max_{\hat{\boldsymbol b}\in\hat{\boldsymbol\eta}_b}\left\{\sum_{k=1}^M\sum_{m\in\mathbb M}\upsilon_{m,k} (  b_{1,m}^k\cdot N^{k}+b_{0,m}^k)\delta t\right\} \nonumber
\\
\label{secondtolast}
=~&\min_{\boldsymbol \upsilon\in V}\left\{\sum_{m\in\mathbb M}\left[\max_{\hat{\boldsymbol b}\in\hat{\boldsymbol\eta}_{b}}\sum_{k=1}^M \upsilon_{m,k}(  b_{1,m}^k\cdot N^{k}+b_{0,m}^k)\delta t\right]\right\}
\end{align}

Consider the inner problem of \eqref{secondtolast} for given $\boldsymbol \upsilon\in V$ and $m\in\mathbb M$:
\begin{align}
\max_{\hat{\boldsymbol b}\in\hat{\boldsymbol\eta}_{b}}\sum_{k=1}^M \upsilon_{m,k}(  b_{1,m}^k\cdot N^{k}+b_{0,m}^k)\delta t.\nonumber
\end{align}
We readily see that its dual formulation is given as (note that $\boldsymbol \upsilon$ and $m$ are fixed here):
\begin{align*}
G^*(\boldsymbol\upsilon,\,m)~\doteq~\min ~~&\sum_{k=1}^M\left( - L_{1,m}\gamma_{1,m}^k+\beta_{1,m}^k U_{1,m} + U_{0,m}\beta_{0,m}^k-L_0\gamma_{0,m}^k\right)+\theta_m\frac{M U_{1,m}}{\sigma_m} 
\\
s.t.~~&-\gamma_{1,m}^k+\beta_{1,m}^k+\theta_m~=~ N^k \upsilon_{m,k}\delta t\qquad \forall 1\leq k\leq M 
\\
&-\gamma_{0,m}^k+\beta_{0,m}^k~=~\upsilon_{m,k}\delta t\qquad \forall 1\leq k\leq M\nonumber
\\
 & \gamma_{1,m}^k~\geq~ 0,~\beta_{1,m}^k~\geq~ 0, ~\gamma_{0,m}^k~\geq~ 0,~\beta_{0,m}^k~\geq~ 0\qquad \forall  1\leq k\leq M 
\\
& \theta_m~\geq~ 0
\end{align*}
By strong duality, 
\begin{equation}\label{strong duality linear program 1}
G^*(\boldsymbol\upsilon,\,m)~=~\max_{\hat{\boldsymbol b}\in\hat{\boldsymbol\eta}_{b}}\sum_{k=1}^M \upsilon_{m,k}(  b_{1,m}^k\cdot N^{k}+b_{0,m}^k)\delta t
\end{equation}
\noindent Therefore, by further invoking \eqref{secondtolast}, constraint \eqref{concave discrete piecewise linear} is equivalent to 
\begin{equation}\label{B.5 to use}
E~\geq~ \max_{\hat{\boldsymbol b}\in\hat{\boldsymbol\eta}_b}\min_{\boldsymbol \upsilon\in V} \mathcal F(\hat{\boldsymbol b},~\boldsymbol \upsilon)~=~\min_{\boldsymbol \upsilon\in V}\left\{\sum_{m\in\mathbb M}G^*(\boldsymbol\upsilon,\,m)\right\}
\end{equation}
In view of the equality in \eqref{some interm problem}, we know that there exists $(\upsilon_{m,k}^*)\in\{0,~1\}^{\vert\mathbb M\vert\times M}\cap V$ such that
\begin{equation}\label{B.6 to use}
\max_{\hat{\boldsymbol b}\in\hat{\boldsymbol\eta}_b}\min_{\boldsymbol \upsilon\in V} \mathcal F(\hat{\boldsymbol b},~\boldsymbol \upsilon)~=~\max_{\hat{\boldsymbol b}\in\hat{\boldsymbol\eta}_b}\mathcal F(\hat{\boldsymbol b},~(\upsilon_{m,k}^*))
\end{equation}
Combining  \eqref{strong duality linear program 1},  \eqref{B.5 to use}, and \eqref{B.6 to use},  we may continue to get
\begin{align*}
~&\min_{\boldsymbol \upsilon\in V}\left\{\sum_{m\in\mathbb M}  G^*(\boldsymbol\upsilon,\,m)\right\}~=~\max_{\hat{\boldsymbol b}\in\hat{\boldsymbol\eta}_b}\min_{\boldsymbol \upsilon\in V} \mathcal F(\hat{\boldsymbol b},~\boldsymbol \upsilon)~=~\max_{\hat{\boldsymbol b}\in\hat{\boldsymbol\eta}_b}\mathcal F(\hat{\boldsymbol b},~(\upsilon_{m,k}^*))
\\
=~&\max_{\hat{\boldsymbol b}\in\hat{\boldsymbol\eta}_b}\left\{\sum_{m\in\mathbb M}\sum_{k=1}^M\upsilon_{m,k}^* (  b_{1,m}^k\cdot N^{k}+b_{0,m}^k)\delta t\right\} ~=~\sum_{m\in\mathbb M}\max_{\hat{\boldsymbol b}\in\hat{\boldsymbol\eta}_b}\left\{\sum_{k=1}^M\upsilon_{m,k}^* (  b_{1,m}^k\cdot N^{k}+b_{0,m}^k)\delta t\right\}
\\
=~&\sum_{m\in\mathbb M} G^*((\upsilon^*_{m,k}),\,m)
\end{align*}
Recall that $(\upsilon_{m,k}^*)\in\{0,~1\}^{\vert\mathbb M\vert\times M}\cap V$. Then the above equality indicates that  
$$
\min_{\boldsymbol \upsilon\in V}\left\{\sum_{m\in\mathbb M} G^*(\boldsymbol\upsilon,\,m)\right\}~=~\min_{\boldsymbol \upsilon\in V\cap\{0,~1\}^{\vert\mathbb M\vert\times M}}\left\{\sum_{m\in\mathbb M}G^*(\boldsymbol\upsilon,\,m)\right\}
$$
Invoking \eqref{B.5 to use} again, we equivalently write constraint \eqref{concave discrete piecewise linear} as
\begin{equation}\label{last but one}
\min_{\boldsymbol \upsilon\in V\cap\{0,~1\}^{\vert\mathbb M\vert\times M}}\left\{\sum_{m\in\mathbb M}G^*(\boldsymbol\upsilon,\,m)\right\}~\leq~ E
\end{equation}
\noindent which is the same as requiring that there exists some $\boldsymbol \upsilon\in V\cap\{0,~1\}^{\vert\mathbb M\vert\times M}$ such that $\sum_{m\in\mathbb M}G^*(\boldsymbol\upsilon,\,m)\leq E$. In addition, notice that $V\cap\{0,~1\}^{\vert\mathbb M\vert\times M}=\{(\upsilon_{m,k})\in\{0,~1\}^{\vert\mathbb M\vert\times M}:~\sum_{m\in\mathbb M}\upsilon_{m,k}=1,~\forall 1\leq k\leq M\}$. Combining this with the definition of $G^*(\boldsymbol\upsilon,\,m)$, we see that constraint \eqref{concave discrete piecewise linear} is equivalent to the following system.
\begin{align}
\begin{split}
&\sum_{m\in\mathbb M}\left\{-\sum_{k=1}^M L_{1,m}\gamma_{1,m}^k+\sum_{k=1}^M\beta_{1,m}^k U_{1,m}+\theta_m\frac{M U_{1,m}}{\sigma_m}+\sum_{k=1}^M U_{0,m}\beta_{0,m}^k-\sum_{k=1}^ML_0\gamma_{0,m}^k\right\}~\leq~ E
\\&-\gamma_{1,m}^k+\beta_{1,m}^k+\theta_m~=~ N^k \upsilon_{m,k}\delta t\qquad \forall 1\leq k\leq M,~m\in\mathbb M
\\&-\gamma_{0,m}^k+\beta_{0,m}^k~=~\upsilon_{m,k}\delta t \qquad \forall 1\leq k\leq M,~m\in\mathbb M
\\ & \gamma_{1,m}^k~\geq~ 0,~\beta_{1,m}^k~\geq~ 0, ~\gamma_{0,m}^k~\geq ~0,~\beta_{0,m}^k~\geq~ 0\qquad \forall  1\leq k\leq M,~m\in\mathbb M
\\& \sum_{m\in\mathbb M}\upsilon_{m,k}~=~1\qquad \forall  1\leq k\leq M
\\& \upsilon_{m,k}\in\{0,~1\}\qquad \forall  1\leq k\leq M,~m\in\mathbb M
\\& \theta_m~\geq~ 0\qquad~m\in\mathbb M
\end{split}\label{nonlinear term}
\end{align}
Observe that the second line of \eqref{nonlinear term}  involves a nonlinear term, which can be replaced with linear constraints with mixed integers using  the ``Big-M'' method. The resulting mixed integer linear constraints are presented in the second and the third lines of \eqref{concave piecewise 1}. 
\end{proof}

\bibliographystyle{model2-names}
\bibliography{<your-bib-database>}

\begin{thebibliography}{00}
\bibitem[Lo, 1999c]{Lo 1999} Lo, H., A dynamic traffic assignment formulation that encapsulates the cell-transmission model. Transportation and Traffic Theory, 327-350.
\end{thebibliography}



\end{document}